\documentclass[11pt,reqno]{amsart}

\usepackage{geometry}
\usepackage{color}
\usepackage{amssymb}
\usepackage{amsmath}
\usepackage{arydshln}

\def\BibTeX{{\rm B\kern-.05em{\sc i\kern-.025em b}\kern-.08em
    T\kern-.1667em\lower.7ex\hbox{E}\kern-.125emX}}

\newtheorem{thm}{Theorem}[section]
\newtheorem{cor}[thm]{Corollary}
\newtheorem{lem}[thm]{Lemma}
\newtheorem{prop}[thm]{Proposition}
\newtheorem{exmp}[thm]{Example}
\newtheorem{rem}[thm]{Remark}
\newtheorem*{claim*}{Claim}
\theoremstyle{definition}
\newtheorem{defn}{Definition}[section]

\numberwithin{equation}{section}

\usepackage{hyperref}
\usepackage{bbm}
\usepackage{mathrsfs}

\usepackage{pdflscape}

\begin{document}
\title[Ergodicity of affine processes on $\mathbb{S}_d^+$]{Ergodicity of affine processes on the cone of symmetric positive semidefinite matrices}

\author{Martin Friesen}
\address[Martin Friesen]{Fakult\"at f\"ur Mathematik und Naturwissenschaften\\ Bergische Universit\"at Wuppertal\\ 42119 Wuppertal, Germany}
\email{friesen@math.uni-wuppertal.de}

\author[Peng Jin]{Peng Jin\textsuperscript{*}}
\thanks{\textsuperscript{*}Peng Jin is supported by the STU Scientific Research Foundation for Talents (No. NTF18023).}
\address[Peng Jin]{Department of Mathematics \\ Shantou University \\ Shantou, Guangdong 515063, China}
\email{pjin@stu.edu.cn}

\author{Jonas Kremer}
\address[Jonas Kremer]{Fakult\"at f\"ur Mathematik und Naturwissenschaften\\ Bergische Universit\"at Wuppertal\\ 42119 Wuppertal, Germany}
\email{kremer@math.uni-wuppertal.de}

\author{Barbara R\"udiger}
\address[Barbara R\"udiger]{Fakult\"at f\"ur Mathematik und Naturwissenschaften\\ Bergische Universit\"at Wuppertal\\ 42119 Wuppertal, Germany}
\email{ruediger@uni-wuppertal.de}

\date{\today}


\begin{abstract}
This article investigates the long-time behavior of conservative affine
processes on the cone of symmetric positive semidefinite $d\times d$-matrices.
In particular, for conservative and subcritical affine processes on
this cone we show that a finite $\log$-moment of
the state-independent jump measure is sufficient for the existence
of a unique limit distribution. Moreover, we study the convergence
rate of the underlying transition kernel to the limit distribution:
firstly, in a specific metric induced by the Laplace transform and
secondly, in the Wasserstein distance under a first moment assumption
imposed on the state-independent jump measure and an additional condition
on the diffusion parameter.
\end{abstract}

\allowdisplaybreaks

\maketitle

\section{Introduction}

An affine process on the cone of symmetric positive semidefinite $d\times d$-matrices
$\mathbb{S}_{d}^{+}$ is a stochastically continuous Markov process
taking values in $\mathbb{S}_{d}^{+}$, whose $\log$-Laplace transform
depends in an affine way on the initial state of the process. \textit{Affine
processes} on the state space $\mathbb{S}_{d}^{+}$
are first systematically studied in the seminal article of Cuchiero
\textit{et al.} \cite{MR2807963}. In their work, the generator of
an $\mathbb{S}_{d}^{+}$-valued affine process is completely characterized
through a set of \emph{admissible parameters}, and the related \emph{generalized
Ricccati equations} %, inherent in the framework of affine processes,
are investigated. Subsequent developments complementing the results
of \cite{MR2807963} can be found in \cite{MR3313754,MR2956112,MR2819242,2018arXiv181110542M}.
Note that the notion of affine processes is not restricted to the
state space $\mathbb{S}_{d}^{+}$. For affine processes
on other finite-dimensional cones, particularly the canonical one
$\mathbb{R}_{+}^{m}\times\mathbb{R}^{n}$, we refer to \cite{MR2284011,MR3340375,MR3540486,MR2243880,MR1994043,MR2648460,MR2931348,MR3313754,MR2851694}.
We remark that the above list is, by far, not complete.

The importance of $\mathbb{S}_{d}^{+}$-valued affine processes has
been demonstrated by their rapidly growing applications in mathematical
finance. In particular, they provide natural models for the evolution
of the covariance matrix of multi-asset prices that exhibit random
dependence, for instance, the Wishart process \cite{MR1132135},
the jump-type Wishart process \cite{LT}, and a certain class of matrix-valued
Ornstein-Uhlenbeck processes driven by L\'evy subordinators \cite{MR2353270}.
Among them, the Wishart process is the most popular one, and it has
been successfully applied to generalize the well-known Heston model
\cite{Heston} to multi-asset setting, see also
\cite{MR3113191,MR3787835,MR3257351,Fonseca2007,MR3011743,MR3252805,MR2723611,MR2785429,MR3564921}.
The jump-type Wishar process as introduced by Leippold and Trojani \cite{LT}
allows jumps which help the model to fit better to real world interest
rates or volatility of multi-asset prices. In \cite{LT} the jump-type
Wishart process is used in multi-variate option pricing, fixed-income
models and dynamic portfolio choice. For a more detailed review on
financial application of affine processes on $\mathbb{S}_{d}^{+}$
we refer to the introduction of \cite{MR2807963}, see also the references
therein.

In this article we investigate the long-time behavior of affine processes
on $\mathbb{S}_{d}^{+}$. First, we study the existence of limit distributions
for these processes. This problem was studied for particular $\mathbb{S}_{d}^{+}$-valued
affine models by Alfonsi \textit{et al.} \cite{MR3549707} in the
case of Wishart processes, while Barndorff-Nielsen and Stelzer \cite{MR2353270}
studied matrix-valued Ornstein-Uhlenbeck processes driven by L\'evy
subordinators. Our main result (see Theorem \ref{thm:existence of limit distribution}
below) is applicable to general conservative, subcritical affine processes
on $\mathbb{S}_{d}^{+}$, and therefore covers the aforementioned
results. Having established the existence of a unique limit distribution
for affine processes on $\mathbb{S}_{d}^{+}$, our next aim is to
study the convergence rate of the underlying transition probability
to the limit distribution in a suitably chosen metric, for instance,
the Wasserstein or total variation distance. While exponential ergodicity
in total variation has been investigated very recently by Mayerhofer
\textit{et al.} \cite{2018arXiv181110542M}, we use two other metrics
in the present article: the Wasserstein-1-distance\footnote{Also known as the Kantorovich-Rubinstein distance.}
and a metric induced by the Laplace transform. We also provide sufficient
conditions for exponential ergodicity with respect to these two metrics.

The long-time behavior of general affine processes has previously
been studied in many different settings, see, e.g., \cite{MR3254346,MR2599675,2018arXiv181205402J,MR2779872,MR2922631,MR2760602,MR738769}.
One application of such a study is towards the calibration of affine
models. In the case of the Wishart process, the maximum-likehood estimator
for the drift parameter was recently studied by Alfonsi \textit{et
al.} \cite{MR3549707}. As demonstrated in their article, ergodicity
helps to derive strong consistency and asymptotic normality of the
estimator.

This paper is organized as follows: In Section \ref{sec:def and main results},
we introduce $\mathbb{S}_{d}^{+}$-valued affine processes, formulate
and discuss our main results. The proofs are then given in Sections
\ref{sec:proof of first moment} -- \ref{sec:proof of exp ergodicity in wasserstein distance}.
Finally, Section \ref{sec:applications} is dedicated to applications
of our results to specific affine models often used in finance.

\section{Main results}

\label{sec:def and main results}

In terms of terminology, we mainly follow the coordinate free notation
used in Mayerhofer \cite{MR2956112} and Keller-Ressel and Mayerhofer
\cite{MR3313754}.

Let $d\geq2$ and denote by $\mathbb{S}_{d}$ the space of symmetric
$d\times d$ matrices equipped with the scalar product $\langle x,y\rangle=\mathrm{tr}(xy)$,
where $\mathrm{tr(\cdot)}$ denotes the trace of a matrix. Accordingly,
$\Vert\cdot\Vert$ is the induced norm on $\mathbb{S}_{d}$, that
is, $\Vert x\Vert:=\langle x,x\rangle^{1/2}$. Note that $\Vert\cdot\Vert$
is the well-known Frobenius norm. We list some properties of the trace
and its induced norm in Appendix \ref{sec:matrix calculus} which
are repeatedly used in the remainder of the article. Denote by $\mathbb{S}_{d}^{+}$
(resp. $\mathbb{S}_{d}^{++}$) the cone of symmetric and positive
semidefinite (resp. positive definite) real $d\times d$ matrices.
We write $x\preceq y$ if $y-x\in\mathbb{S}_{d}^{+}$ and $x\prec y$
if $y-x\in\mathbb{S}_{d}^{++}$ for the natural partial and strict
order relation introduced respectively by the cones $\mathbb{S}_{d}^{+}$
and $\mathbb{S}_{d}^{++}$. Let $\mathcal{B}(\mathbb{S}_{d}^{+}\backslash\lbrace0\rbrace)$
be the Borel-$\sigma$-algebra on $\mathbb{S}_{d}^{+}\backslash\lbrace0\rbrace$.
An \textit{$\mathbb{S}_{d}^{+}$-valued measure} $\eta$ on $\mathbb{S}_{d}^{+}\backslash\lbrace0\rbrace$
is a $d\times d$-matrix of signed measures on $\mathbb{S}_{d}^{+}\backslash\lbrace0\rbrace$
such that $\eta(A)\in\mathbb{S}_{d}^{+}$ whenever $A\in\mathcal{B}(\mathbb{S}_{d}^{+}\backslash\lbrace0\rbrace)$
with $0\not\in\overline{A}$.

In the following we introduce the notion of admissible parameters
first introduced in Cuchiero \textit{et al.} \cite[Definition 2.3]{MR2807963}.
Here we mainly follow the one given in Mayerhofer \cite[Definition 3.1]{MR2956112},
with a slightly stronger condition on the linear jump coefficient.

\begin{defn}\label{def:admissible parameters} Let $d\geq2$. An
admissible parameter set $(\alpha,b,B,m,\mu)$ consists of:

(i) a linear diffusion coefficient $\alpha\in\mathbb{S}_{d}^{+}$;

(ii) a constant drift $b\in\mathbb{S}_{d}^{+}$ satisfying $b\succeq(d-1)\alpha$;

(iii) a constant jump term: a Borel measure $m$ on $\mathbb{S}_{d}^{+}\backslash\{0\}$
satisfying
\[
\int_{\mathbb{S}_{d}^{+}\backslash\{0\}}\left(\left\Vert \xi\right\Vert \wedge1\right)m\left(\mathrm{d}\xi\right)<\infty;
\]

(iv) a linear jump coefficient $\mu$ which is an $\mathbb{S}_{d}^{+}$-valued,
sigma-finite measure on $\mathbb{S}_{d}^{+}\backslash\{0\}$ satisfying
\[
\int_{\mathbb{S}_{d}^{+}\backslash\{0\}}\left\Vert \xi\right\Vert \mathrm{tr}(\mu)\left(\mathrm{d}\xi\right) < \infty,
\]
where $\mathrm{tr}(\mu)$ denotes the measure induced
by the relation $\mathrm{tr}(\mu)(A):=\mathrm{tr}(\mu(A))$ for all
$A\in\mathcal{B}(\mathbb{S}_{d}^{+}\backslash\lbrace0\rbrace)$
with $0\notin\bar{A}$;

(v) a linear drift $B$, which is a linear map $B:\mathbb{S}_{d}\to\mathbb{S}_{d}$
satisfying
\[
\langle B(x),u\rangle\geq0\quad\text{for all }x,\thinspace u\in\mathbb{S}_{d}^{+}\text{ with }\langle x,u\rangle=0.
\]
\end{defn}

According to our definition, a set of admissible parameters does not
contain parameters corresponding to killing. In addition, our definition
involves a first moment assumption on the linear jump coefficient
$\mu$.

\begin{thm}[\cite{MR2807963}]\label{thm:characterization theorem}
Let $(\alpha,b,B,m,\mu)$ be admissible parameters in the sense of
Definition \ref{def:admissible parameters}. Then there exists a unique
stochastically continuous transition kernel $p_{t}(x,\mathrm{d}\xi)$
such that $p_{t}(x,\mathbb{S}_{d}^{+})=1$ and
\begin{equation}
\int_{\mathbb{S}_{d}^{+}}\mathrm{e}^{-\langle u,\xi\rangle}p_{t}(x,\mathrm{d}\xi)=\exp\left(-\phi(t,u)-\langle\psi(t,u),x\rangle\right),\quad t\geq0,\ \ x,u\in\mathbb{S}_{d}^{+},\label{eq:affine representation}
\end{equation}
where $\phi(t,u)$ and $\psi(t,u)$ in \eqref{eq:affine representation}
are the unique solutions to the generalized Riccati differential equations,
that is, for $u\in\mathbb{S}_{d}^{+}$,
\begin{align}
\frac{\partial\phi(t,u)}{\partial t} & =F\left(\psi(t,u)\right),\quad\phi(0,u)=0,\label{eq:riccati equation for phi}\\
\frac{\partial\psi(t,u)}{\partial t} & =R\left(\psi(t,u)\right),\quad\psi(0,u)=u,\label{eq:riccati equation for psi}
\end{align}
and the functions $F$ and $R$ are given by
\begin{align*}
F(u) & =\langle b,u\rangle-\int_{\mathbb{S}_{d}^{+}\backslash\{0\}}\left(\mathrm{e}^{-\langle u,\xi\rangle}-1\right)m\left(\mathrm{d}\xi\right),\\
R(u) & =-2u\alpha u+B{}^{\top}\left(u\right)-\int_{\mathbb{S}_{d}^{+}\backslash\{0\}}\left(\mathrm{e}^{-\langle u,\xi\rangle}-1\right)\mu\left(\mathrm{d}\xi\right).
\end{align*}
\end{thm}

Here, $B^{\top}$ denotes the adjoint operator on $\mathbb{S}_{d}$
defined by the relation $\langle u,B(\xi)\rangle=\langle B^{\top}(u),\xi\rangle$
for $u,\thinspace\xi\in\mathbb{S}_{d}$. Under the
additional moment condition (iv) of Definition \ref{def:admissible parameters},
we will show in Lemma \ref{lem:space-differentiability of F and R}
below that $R(u)$ is continuously differentiable and thus locally
Lipschitz continuous on $\mathbb{S}_{d}^{+}$. This fact, together
with the absence of parameters according to killing, implies that
the affine process under consideration is indeed conservative (see
\cite[Remark 2.5]{MR2807963}).

\subsection{First moment}

Our first result provides existence and a precise formula for the
first moment of conservative affine processes on $\mathbb{S}_{d}^{+}$.
For this purpose, we define the effective drift
\[
\widetilde{B}(u):=B(u)+\int_{\mathbb{S}_{d}^{+}\backslash\{0\}}\langle\xi,u\rangle\mu\left(\mathrm{d}\xi\right),\quad\text{for all }u\in\mathbb{S}_{d}.
\]
Then note that $\widetilde{B}:\mathbb{S}_{d}\to\mathbb{S}_{d}$ is
a linear map. We define the corresponding semigroup $(\exp(t\widetilde{B}))_{t\geq0}$
by its Taylor series $\exp(t\widetilde{B})(u)=\sum_{n=0}^{\infty}t^{n}/n!\widetilde{B}^{\circ n}(u)$,
where $\widetilde{B}^{\circ n}$ denotes the $n$-times composition
of $\widetilde{B}$. For the remainder of the article we write $\mathbbm{1}$
without an index for the $d\times d$-identity matrix, while $\mathbbm{1}_{A}$
denotes the standard indicator function of a set $A$.

\begin{thm}\label{thm: first moment} Let $p_{t}(x,\mathrm{d}\xi)$
be the transition kernel of an affine process on $\mathbb{S}_{d}^{+}$
with admissible parameters $(\alpha,b,B,m,\mu)$ satisfying
\begin{equation}
\int_{\{\Vert\xi\Vert>1\}}\Vert\xi\Vert m\left(\mathrm{d}\xi\right)<\infty.\label{eq:first moment}
\end{equation}
Then, for each $t\geq0$ and $x\in\mathbb{S}_{d}^{+}$,
\begin{equation}
\int_{\mathbb{S}_{d}^{+}}\xi p_{t}(x,\mathrm{d}\xi)=\mathrm{e}^{t\widetilde{B}}x+\int_{0}^{t}\mathrm{e}^{s\widetilde{B}}\left(b+\int_{\mathbb{S}_{d}^{+}\backslash\lbrace0\rbrace}\xi m(\mathrm{d}\xi)\right)\mathrm{d}s.\label{eq:first moment for affine processes}
\end{equation}
In particular, the first moment exists. \end{thm}

Based on methods of stochastic calculus similar results were obtained
for affine processes with state space $\mathbb{R}_{\geq0}^{m}$ in
\cite[Lemma 3.4]{MR3340375} and on the canonical state space $\mathbb{R}_{\geq0}^{m}\times\mathbb{R}^{n}$
in \cite[Lemma 5.2]{2019arXiv190105815F}. For affine processes on
$\mathbb{R}_{\geq0}$, i.e., continuous-state branching processes
with immigration, and also for the more general class of Dawson-Watanabe
superprocesses an alternative approach based on a fine analysis of
the Laplace transform is provided in \cite{MR2760602}. The latter
approach has clearly the advantage that it is purely analytical and
does not rely on the use of stochastic equations and semimartingale
representations for these processes. We provide in Section \ref{sec:proof of first moment}
a purely analytic proof for Theorem \ref{thm: first moment} as well.

\begin{rem} Note that the transition kernel $p_{t}(x,\cdot)$ with
admissible parameters $(\alpha,b,B,m,\mu)$ is Feller by virtue of
\cite[Theorem 2.4]{MR2807963}. Therefore, there exists a canonical
realization $(X,(\mathbb{P}_{x})_{x\in\mathbb{S}_{d}^{+}})$ of the
corresponding Markov process on the filtered space $(\Omega,\mathcal{F},(\mathcal{F}_{t})_{t\geq0})$,
where $\Omega=\mathbb{D}(\mathbb{S}_{d}^{+})$ is the set of all c\`{a}dl\`{a}g
paths $\omega:\mathbb{R}_{\geq0}\to\mathbb{S}_{d}^{+}$ and $X_{t}(\omega)=\omega(t)$
for $\omega\in\Omega$. Here $(\mathcal{F}_{t})_{t\geq0}$ is the
natural filtration generated by $X$ and $\mathcal{F}=\bigvee_{t\geq0}\mathcal{F}_{t}$.
For $x\in\mathbb{S}_{d}^{+}$, the probability measure $\mathbb{P}_{x}$
on $\Omega$ represents the law of the Markov process $X$ given $X_{0}=x$.
With this notation, under the conditions of Theorem \ref{thm: first moment},
formula \eqref{eq:first moment for affine processes} reads
\[
\mathbb{E}_{x}\left[X_{t}\right]=\mathrm{e}^{t\widetilde{B}}x+\int_{0}^{t}\mathrm{e}^{s\widetilde{B}}\left(b+\int_{\mathbb{S}_{d}^{+}\backslash\lbrace0\rbrace}\xi m(\mathrm{d}\xi)\right)\mathrm{d}s,
\]
where $\mathbb{E}_{x}$ denotes the expectation with respect to $\mathbb{P}_{x}$.
\end{rem}

\subsection{Existence and convergence to the invariant distribution}

In this subsection we formulate our main result. Let $p_{t}(x,\cdot)$
be the transition kernel of an affine process on $\mathbb{S}_{d}^{+}$.
Motivated by Theorem \ref{thm: first moment} it is reasonable to
relate the long-time behavior of the process with the spectrum $\sigma(\widetilde{B})$
of $\widetilde{B}$. More precisely, an affine process on $\mathbb{S}_{d}^{+}$
with admissible parameters $(\alpha,b,B,m,\mu)$ is said to be \textit{subcritical},
if
\begin{equation}
\sup\left\lbrace \mathrm{Re}\thinspace\lambda\in\mathbb{C}\thinspace:\thinspace\lambda\in\sigma\left(\widetilde{B}\right)\right\rbrace <0.\label{eq:delta}
\end{equation}
Under condition \eqref{eq:delta}, it is well-known that there exist constants $M\ge 1$ and $\delta>0$
such that
\begin{equation}
\left\Vert \mathrm{e}^{t\widetilde{B}}\right\Vert \leq M\mathrm{e}^{-\delta t},\quad t\geq0.\label{subcritical}
\end{equation}
The next remark provides a sufficient condition for \eqref{subcritical}.

\begin{rem}\label{rem:equivalence for delta condition} According
to \cite[Theorem 2.7]{2018arXiv181110542M}, \eqref{subcritical} is
satisfied if and only if there exists a $v\in\mathbb{S}_{d}^{++}$
such that $-\widetilde{B}^{\top}(v)\in\mathbb{S}_{d}^{++}$. However,
in many application the linear drift is of the form $\widetilde{B}(x)=\beta x+x\beta^{\top}$,
where $\beta$ is a real-valued $d\times d$-matrix, see Section \ref{sec:applications}.
In this case, it follows from \cite[Corollary 5.1]{2018arXiv181110542M}
that \eqref{subcritical} is satisfied if and only if
\[
\sup\left\lbrace \mathrm{Re}\thinspace\lambda\in\mathbb{C}\thinspace:\thinspace\lambda\in\sigma\left(\beta\right)\right\rbrace <0,
\]
which in turn holds true if and only if there exists one $v\in\mathbb{S}_{d}^{++}$
such that $-(\beta^{\top}v+v\beta)\in\mathbb{S}_{d}^{++}$. \end{rem}

Let $\mathcal{P}(\mathbb{S}_{d}^{+})$ be the space of all Borel probability
measures on $\mathbb{S}_{d}^{+}$. We call $\pi\in\mathcal{P}(\mathbb{S}_{d}^{+})$
an invariant distribution, if
\[
\int_{\mathbb{S}_{d}^{+}}p_{t}(x,\mathrm{d}\xi)\pi(\mathrm{d}x)=\pi(\mathrm{d}\xi),\qquad t\geq0.
\]
The following is our main result.

\begin{thm}\label{thm:existence of limit distribution} Let $p_{t}(x,\mathrm{d}\xi)$
be the transition kernel of a subcritical affine process on $\mathbb{S}_{d}^{+}$
with admissible parameters $(\alpha,b,B,m,\mu)$. Suppose that the
measure $m$ satisfies
\begin{align}
\int_{\{\Vert\xi\Vert>1\}}\log\left\Vert \xi\right\Vert m\left(\mathrm{d}\xi\right)<\infty.\label{LOGMOMENT}
\end{align}
Then there exists a unique invariant distribution $\pi$. Moreover,
$p_{t}(x,\cdot)\to\pi$ weakly as $t\to\infty$ for each $x\in\mathbb{S}_{d}^{+}$
and $\pi$ has Laplace transform
\begin{equation}
\int_{\mathbb{S}_{d}^{+}}\mathrm{e}^{-\langle u,x\rangle}\pi\left(\mathrm{d}x\right)=\exp\left(-\int_{0}^{\infty}F\left(\psi(s,u)\right)\mathrm{d}s\right),\quad u\in\mathbb{S}_{d}^{+}.\label{eq:laplace transform of pi}
\end{equation}
\end{thm}

The proof of Theorem \ref{thm:existence of limit distribution} is
postponed to Section \ref{sec:proof of ergodicity}. Let us make a
few comments. Note that in dimension $d=1$ it holds $\mathbb{S}_{1}^{+}=\mathbb{R}_{\geq0}$
and affine processes on this state space coincide with the class of
continuous-state branching processes with immigration introduced by
Kawazu and Watanbe \cite{MR0290475}. In this case, the long-time
behavior has been extensively studied in the articles \cite[Theorem 3.16]{MR2390186},
\cite[Theorem 2.6]{MR2922631}, and the monograph \cite[Theorem 3.20 and Corollary 3.21]{MR2760602}.
This is why we restrict ourselves to the case $d\geq2$. Theorem \ref{thm:existence of limit distribution}
establishes sufficient conditions for the existence, uniqueness, and
convergence to the invariant distribution. For affine processes on
the canonical state space $\mathbb{R}_{\geq0}^{m}\times\mathbb{R}^{n}$
a similar statement was recently shown in \cite{2018arXiv181205402J}.

For dimension $d=1$ it is known that \eqref{LOGMOMENT} is not only
sufficient, but also necessary for the convergence to some limiting
distribution, see, e.g., \cite[Theorem 3.20 and Corollary 3.21]{MR2760602}.
To our knowledge, extensions of this result to higher dimensional
state space has not yet been obtained. In this context, we have the
following partial result for subcritical affine processes on $\mathbb{S}_{d}^{+}$.

\begin{prop}\label{prop:necessity for ergodicity} Let $p_{t}(x,\mathrm{d}\xi)$
be the transition kernel of a subcritical affine process on $\mathbb{S}_{d}^{+}$
with admissible parameters $(\alpha,b,B,m,\mu)$. Suppose that there
exists $x\in\mathbb{S}_{d}^{+}$ and $\pi\in\mathcal{P}(\mathbb{S}_{d}^{+})$
such that $p_{t}(x,\cdot)\to\pi$ weakly as $t\to\infty$. If $\alpha=0$
and there exists a constant $K>0$ satisfying
\begin{align}
K\xi+B(\xi)\succeq0,\quad\xi\in\mathbb{S}_{d}^{+},\label{eq:00}
\end{align}
then \eqref{LOGMOMENT} holds. \end{prop}

We note that any linear map $B:\mathbb{S}_d\to\mathbb{S}_d$ which leaves $\mathbb{S}_d^+$ invariant satisfies condition \eqref{eq:00} for each $K>0$. As an example of such a map,  let $B(x)=\beta x\beta^{\top}$ for $x\in\mathbb{S}_d$, where $\beta$ is a real-valued invertible $d\times d$-matrix. Obviously,  $B$ defined in this way is admissible in the sense of Definition \ref{def:admissible parameters} and $B(\mathbb{S}_d^+)=\mathbb{S}_d^+$. Moreover, in view of \cite[Theorem 2]{MR0173678},  any linear map that leaves $\mathbb{S}_d^+$ invariant must be of this form.

In order to prove Theorem \ref{thm:existence of limit distribution}
and Proposition \ref{prop:necessity for ergodicity} we first establish
in Section \ref{sec:estimates on psi} precise lower and upper bounds
for $\psi(t,u)$. Since in dimension $d\geq2$ different components
of the process interact through the drift $B$ in a nontrivial manner
on $\mathbb{S}_{d}^{+}$, the proof of the lower bound is deduced
from the additional conditions $\alpha=0$ and \eqref{eq:00}, which
guarantees that these components are coupled in a \textit{well-behaved
way}.

We close this section with a useful moment result regarding the invariant
distribution.

\begin{cor}\label{cor:first moment of pi} Let $p_{t}(x,\mathrm{d}\xi)$
be the transition kernel of a subcritical affine process on $\mathbb{S}_{d}^{+}$
with admissible parameters $(\alpha,b,B,m,\mu)$ satisfying \eqref{eq:first moment}.
Let $\pi$ be the unique invariant distribution. Then
\[
\lim_{t\to\infty}\int_{\mathbb{S}_{d}^{+}}yp_{t}(x,\mathrm{d}y)=\int_{\mathbb{S}_{d}^{+}}y\pi(\mathrm{d}y)=\int_{0}^{\infty}\mathrm{e}^{s\widetilde{B}}\left(b+\int_{\mathbb{S}_{d}^{+}\backslash\lbrace0\rbrace}\xi m(\mathrm{d}\xi)\right)\mathrm{d}s.
\]
\end{cor}

\subsection{Study of convergence rate}

Noting that $\delta$ defined by \eqref{subcritical} is supposed
to be strictly positive, we will see that it appears naturally in
the rate of convergence towards the invariant distribution. In order
to measure this rate of convergence we introduce
\[
d_{L}(\eta,\nu):=\sup_{u\in\mathbb{S}_{d}^{+}\setminus\{0\}}\frac{1}{\Vert u\Vert}\left|\int_{\mathbb{S}_{d}^{+}}\mathrm{e}^{-\langle u,x\rangle}\eta(\mathrm{d}x)-\int_{\mathbb{S}_{d}^{+}}\mathrm{e}^{-\langle u,x\rangle}\nu(\mathrm{d}x)\right|,\quad\eta,\thinspace\nu\in\mathcal{P}(\mathbb{S}_{d}^{+}).
\]
Note that this supremum is not necessarily finite.
However, it is finite for elements of
\[
\mathcal{P}_{1}(\mathbb{S}_{d}^{+})=\left\lbrace \varrho\in\mathcal{P}(\mathbb{S}_{d}^{+})\thinspace:\thinspace\int_{\mathbb{S}_{d}^{+}}\Vert x\Vert\varrho(\mathrm{d}x)<\infty\right\rbrace .
\]
Then it is easy to see that $d_L$ is a metric on $\mathcal{P}_1(\mathbb{S}_{d}^{+})$; moreover, $\big(\mathcal{P}_1(\mathbb{S}_{d}^{+}),d_L\big)$ is complete. Using well-known properties of Laplace transforms, it can be shown
that convergence with respect to $d_{L}$ is stronger than weak convergence. 
The next result
provides an exponential rate in $d_{L}$ distance.

\begin{thm}\label{thm:convergence dL} Let $p_{t}(x,\mathrm{d}\xi)$
be the transition kernel of a subcritical affine process on $\mathbb{S}_{d}^{+}$
with admissible parameters $(\alpha,b,B,m,\mu)$. Suppose that \eqref{LOGMOMENT}
holds and denote by $\pi$ the unique invariant distribution. Then
there exists a constant $C>0$ such that
\begin{equation}
d_{L}\left(p_{t}(x,\cdot),\pi\right)\leq C\left(1+\Vert x\Vert\right)\mathrm{e}^{-\delta t},\quad t\geq0,\ \ x\in\mathbb{S}_{d}^{+}.\label{eq:convergence speed of pi}
\end{equation}
\end{thm}

The proof of this result is given in Section \ref{sec:proof of convergence in dL}. Although under the given conditions $p_t(x,\cdot)$ and $\pi$ do not necessarily belong to $\mathcal{P}_1(\mathbb{S}_d^+)$, the proof of \eqref{eq:convergence speed of pi} implies that $d_L(p_t(x,\cdot), \pi)$ is well-defined.

We turn to investigate the convergence rate from the affine transition
kernel to the invariant distribution in the Wasserstein-1-distance
introduced below. Given $\varrho,\thinspace\widetilde{\varrho}\in\mathcal{P}_{1}(\mathbb{S}_{d}^{+})$,
a coupling $H$ of $(\varrho,\widetilde{\varrho})$ is a Borel probability
measure on $\mathbb{S}_{d}^{+}\times\mathbb{S}_{d}^{+}$ which has
marginals $\varrho$ and $\widetilde{\varrho}$, respectively. We
denote by $\mathcal{H}(\varrho,\widetilde{\varrho})$ the collection
of all such couplings. We define the \emph{Wasserstein distance} on
$\mathcal{P}_{1}(\mathbb{S}_{d}^{+})$ by
\[
W_{1}\left(\varrho,\widetilde{\varrho}\right)=\inf\left\lbrace \int_{\mathbb{S}_{d}^{+}\times\mathbb{S}_{d}^{+}}\Vert x-y\Vert H\left(\mathrm{d}x,\mathrm{d}y\right)\thinspace:\thinspace H\in\mathcal{H}\left(\varrho,\widetilde{\varrho}\right)\right\rbrace .
\]
Since $\varrho$ and $\widetilde{\varrho}$ belong to $\mathcal{P}_{1}(\mathbb{S}_{d}^{+})$,
it holds that $W_{1}(\varrho,\widetilde{\varrho})$ is finite. According
to \cite[Theorem 6.16]{MR2459454}, we have that $(\mathcal{P}(\mathbb{S}_{d}^{+}),W_{1})$
is a complete separable metric space. Exponential ergodicity in different
Wasserstein distances for affine processes on the canonical state
space $\mathbb{R}_{+}^{m}\times\mathbb{R}^{n}$ was very recently
studied in \cite{2019arXiv190105815F}. Below we provide a corresponding
result for affine processes on $\mathbb{S}_{d}^{+}$.

\begin{thm}\label{th:wasserstein ergodicity} Let $p_{t}(x,\mathrm{d}\xi)$
be the transition kernel of a subcritical affine process on $\mathbb{S}_{d}^{+}$
with admissible parameters $(\alpha,b,B,m,\mu)$ satisfying \eqref{eq:first moment}.
If $\alpha=0$, then
\begin{equation}
W_{1}\left(p_{t}(x,\cdot),\pi\right)\leq\sqrt{d}Me^{-\delta t}\left(\|x\|+\int_{\mathbb{S}_{d}^{+}}\|y\|\pi(\mathrm{d}y)\right),\quad t\geq0,\ \ x\in\mathbb{S}_{d}^{+}.\label{eq:05}
\end{equation}
\end{thm}

The proof of Theorem \ref{th:wasserstein ergodicity} is given in
Section \ref{sec:proof of exp ergodicity in wasserstein distance}
which largely follows some ideas of \cite{2019arXiv190105815F}. In
contrast to the latter work, for the study of affine processes on
$\mathbb{S}_{d}^{+}$ we encounter two additional difficulties:
\begin{itemize}
\item It is still an open problem whether each affine process on $\mathbb{S}_{d}^{+}$
can be obtained as a strong solution to a certain stochastic equation
driven by Brownian motions and Poisson random measures. We refer the
reader to \cite{MR2819242} for some related results. In addition,
we do not know if a comparison principle for such processes would
be available.
\item Following \cite{2019arXiv190105815F}, one important step in the proof
of Theorem 2.7 therein is based on the decomposition $p_{t}(x,\cdot)=r_{t}(x,\cdot)\ast p_{t}(0,\cdot)$,
where $r_{t}(x,\cdot)$ is the transition kernel of an affine process
on $\mathbb{S}_{d}^{+}$ whose Laplace transform is given by
\[
\int_{\mathbb{S}_{d}^{+}}\mathrm{e}^{-\langle u,\xi\rangle}r_{t}(x,\mathrm{d}\xi)=\exp\left(-\langle\psi(t,u),x\rangle\right),
\]
that is, $r_{t}(x,\cdot)$ should have admissible parameters $(\alpha,b=0,B,m=0,\mu)$.
Unfortunately, such transition kernel $r_{t}(x,\cdot)$ is well-defined
if and only if $(\alpha,b=0,B,m=0,\mu)$ are admissible parameters
in the sense of Definition \ref{def:admissible parameters}. This
in turn is true if and only if $\alpha=0$ which is a consequence
of the particular structure of the boundary $\mathbb{S}_{d}^{+}\backslash\mathbb{S}_{d}^{++}$.
\end{itemize}

\section{Proof of Theorem \ref{thm: first moment}}

\label{sec:proof of first moment}

In this section we study the first moment of a conservative affine
process on $\mathbb{S}_{d}^{+}$. In particular, we prove Theorem
\ref{thm: first moment}. Essential to the proof is the space-differentiability
of the functions $F$ and $R$ as well as $\phi$ and $\psi$. To
simplify the notation we introduce $L(\mathbb{S}_{d},\mathbb{S}_{d})$
as the space of all linear operators $\mathbb{S}_{d}\to\mathbb{S}_{d}$,
and similarly $L(\mathbb{S}_{d},\mathbb{R})$ stands
for the space of all linear functionals $\mathbb{S}_{d}\to\mathbb{R}$.
For a function $G:\mathbb{S}_{d}\to\mathbb{S}_{d}$ we denote its
derivative at $u\in\mathbb{S}_{d}$, if it exists,
by $DG(u)\in L(\mathbb{S}_{d},\mathbb{S}_{d})$. Similarly, we denote
the derivative of $H:\mathbb{S}_{d}\to\mathbb{R}$ by $DH(u)\in L(\mathbb{S}_{d},\mathbb{R})$.
We equip $L(\mathbb{S}_{d},\mathbb{S}_{d})$ and $L(\mathbb{S}_{d},\mathbb{R})$
with the corresponding norm
\[
\left\Vert DG(u)\right\Vert =\sup_{\Vert x\Vert=1}\left\Vert DG(u)(x)\right\Vert \quad\text{and}\quad\left\Vert DH(u)\right\Vert =\sup_{\Vert x\Vert=1}\left\Vert DH(u)(x)\right\Vert .
\]

Let $F$ and $R$ be as in Theorem \ref{thm:characterization theorem}.
According to \cite[Lemma 5.1]{MR2807963} the function
$R$ is analytic on $\mathbb{S}_{d}^{++}$. Below we study the differentiability
of $F$ and $R$  on the entire cone $\mathbb{S}_{d}^{+}$.

We first give a lemma that slightly extends \cite[Lemma 3.3]{MR2956112}.

\begin{lem}
\label{lem: triangle ineq.} Let $g$ be a measurable
function on $\mathbb{S}_{d}^{+}$ with $\int_{\mathbb{S}_{d}^{+}\backslash\lbrace0\rbrace}\left\vert g(\xi)\right\vert \mathrm{tr}(\mu)(\mathrm{d}\xi)<\infty$.
Then $\int_{\mathbb{S}_{d}^{+}\backslash\lbrace0\rbrace}g(\xi)\mu(\mathrm{d}\xi)$
is finite and
\[
\left\Vert \int_{\mathbb{S}_{d}^{+}\backslash\lbrace0\rbrace}g(\xi)\mu(\mathrm{d}\xi)\right\Vert \le\int_{\mathbb{S}_{d}^{+}\backslash\lbrace0\rbrace}\left\vert g(\xi)\right\vert \mathrm{tr}(\mu)(\mathrm{d}\xi).
\]
\end{lem}

\begin{proof}
Let $\mu=(\mu_{ij})$ and $\mu_{ij}=\mu_{ij}^{+}-\mu_{ij}^{-}$
be the Jordan decomposition of $\mu_{ij}$. Suppose $\int_{\mathbb{S}_{d}^{+}\backslash\lbrace0\rbrace}\left\vert g(\xi)\right\vert \mathrm{tr}(\mu)(\mathrm{d}\xi)<\infty$.
Then \cite[Lemma 3.3]{MR2956112} implies that $\int_{\mathbb{S}_{d}^{+}\backslash\lbrace0\rbrace}|g(\xi)|\mu(\mathrm{d}\xi)$
is finite and
\[
\left\Vert \int_{\mathbb{S}_{d}^{+}\backslash\lbrace0\rbrace}|g(\xi)|\mu(\mathrm{d}\xi)\right\Vert \le\int_{\mathbb{S}_{d}^{+}\backslash\lbrace0\rbrace}\left\vert g(\xi)\right\vert \mathrm{tr}(\mu)(\mathrm{d}\xi).
\]
Since the $ij$-th entry of $\int_{\mathbb{S}_{d}^{+}\backslash\lbrace0\rbrace}|g(\xi)|\mu(\mathrm{d}\xi)$
is given by
\[
\int_{\mathbb{S}_{d}^{+}\backslash\lbrace0\rbrace}|g(\xi)|\mu_{ij}^{+}(\mathrm{d}\xi)-\int_{\mathbb{S}_{d}^{+}\backslash\lbrace0\rbrace}|g(\xi)|\mu_{ij}^{-}(\mathrm{d}\xi),
\]
which is finite, we must have
\[
\int_{\mathbb{S}_{d}^{+}\backslash\lbrace0\rbrace}|g(\xi)|\mu_{ij}^{+}(\mathrm{d}\xi)<\infty\qquad\mbox{and\ensuremath{\qquad}}\int_{\mathbb{S}_{d}^{+}\backslash\lbrace0\rbrace}|g(\xi)|\mu_{ij}^{-}(\mathrm{d}\xi)<\infty,\quad\forall i,j\in\left\{ 1,\ldots d\right\} .
\]
So $\int_{\mathbb{S}_{d}^{+}\backslash\lbrace0\rbrace}g(\xi)\mu(\mathrm{d}\xi)$
is finite. Again by \cite[Lemma 3.3]{MR2956112},
\begin{align*}
\left\Vert \int_{\mathbb{S}_{d}^{+}\backslash\lbrace0\rbrace}g(\xi)\mu(\mathrm{d}\xi)\right\Vert  & =\left\Vert \int_{\mathbb{S}_{d}^{+}\backslash\lbrace0\rbrace}g^{+}(\xi)\mu(\mathrm{d}\xi)-\int_{\mathbb{S}_{d}^{+}\backslash\lbrace0\rbrace}g^{-}(\xi)\mu(\mathrm{d}\xi)\right\Vert \\
 & \le\left\Vert \int_{\mathbb{S}_{d}^{+}\backslash\lbrace0\rbrace}g^{+}(\xi)\mu(\mathrm{d}\xi)\right\Vert +\left\Vert \int_{\mathbb{S}_{d}^{+}\backslash\lbrace0\rbrace}g^{-}(\xi)\mu(\mathrm{d}\xi)\right\Vert \\
 & \le\int_{\mathbb{S}_{d}^{+}\backslash\lbrace0\rbrace}g^{+}(\xi)\mathrm{tr}(\mu)(\mathrm{d}\xi)+\int_{\mathbb{S}_{d}^{+}\backslash\lbrace0\rbrace}g^{-}(\xi)\mathrm{tr}(\mu)(\mathrm{d}\xi)\\
 & \le\int_{\mathbb{S}_{d}^{+}\backslash\lbrace0\rbrace}\left\vert g(\xi)\right\vert \mathrm{tr}(\mu)(\mathrm{d}\xi).
\end{align*}
The lemma is proved.
\end{proof}

\begin{lem}\label{lem:space-differentiability of F and R}
The following statements hold:
\begin{enumerate}
\item[(a)]  For $u\in\mathbb{S}_{d}^{++}$, $h\in\mathbb{S}_{d}$,
we have
\begin{equation}
DR(u)(h)=-2\left(u\alpha h+h\alpha u\right)+B^{\top}(h)+\int_{\mathbb{S}_{d}^{+}\backslash\lbrace0\rbrace}\langle h,\xi\rangle\mathrm{e}^{-\langle u,\xi\rangle}\mu(\mathrm{d}\xi).\label{eq: DR}
\end{equation}
Moreover, through \eqref{eq: DR} $DR(u)$ is continuously extended
to $u\in\mathbb{S}_{d}^{+}$. In particular, $R\in C^{1}(\mathbb{S}_{d}^{+})$
and \eqref{eq: DR} holds true for all $u\in\mathbb{S}_{d}^{+},h\in\mathbb{S}_{d}$.
\item[(b)] If \eqref{eq:first moment} is satisfied, then for
$u\in\mathbb{S}_{d}^{++}$, $h\in\mathbb{S}_{d}$,
\begin{equation}
DF(u)(h)=\langle b,h\rangle+\int_{\mathbb{S}_{d}^{+}\backslash\lbrace0\rbrace}\langle h,\xi\rangle\mathrm{e}^{-\langle u,\xi\rangle}m(\mathrm{d}\xi).\label{eq: DF}
\end{equation}
 Moreover, through \eqref{eq: DR} $DF(u)$ is continuously extended
to $u\in\mathbb{S}_{d}^{+}$. In particular, $F\in C^{1}(\mathbb{S}_{d}^{+})$
and \eqref{eq: DF} holds true for all $u\in\mathbb{S}_{d}^{+},h\in\mathbb{S}_{d}$.
\end{enumerate}
\end{lem}

\begin{proof} (a) Let $u\in\mathbb{S}_{d}^{++}$.
Consider $h\in\mathbb{S}_{d}$ with sufficiently small $\|h\|$ such
that $u+h\in\mathbb{S}_{d}^{+}$. An easy calculation shows that
\[
R(u+h)-R(u)=DR(u)(h)+r(u,h),
\]
where
\[
r(u,h):=-2h\alpha h+\int_{\mathbb{S}_{d}^{+}\backslash\lbrace0\rbrace}\mathrm{e}^{-\langle u,\xi\rangle}\left(1-\mathrm{e}^{-\langle h,\xi\rangle}-\langle h,\xi\rangle\right)\mu(\mathrm{d}\xi).
\]
Let us prove that $\lim_{0\not=\Vert h\Vert\to0}\Vert r(u,h)\Vert/\Vert h\Vert=0$.
Assume $\Vert h\Vert\not=0$. First, note that
\[
\frac{\Vert2h\alpha h\Vert}{\Vert h\Vert}\leq2\Vert\alpha\Vert\frac{\Vert h\Vert^{2}}{\Vert h\Vert}\leq2\Vert\alpha\Vert\Vert h\Vert.
\]
Let $M>0$. For $\Vert\xi\Vert\leq M$, we have
\begin{align}
\left|\mathrm{e}^{-\langle u,\xi\rangle}\left(1-\mathrm{e}^{-\langle h,\xi\rangle}-\langle h,\xi\rangle\right)\right|
 & =\left|\langle h,\xi\rangle\left(\int_{0}^{1}\mathrm{e}^{-\langle u+sh,\xi\rangle}\mathrm{d}s-\mathrm{e}^{-\langle u,\xi\rangle}\right)\right|\nonumber \\
 & =\left|\langle h,\xi\rangle\right|\cdot\left|\int_{0}^{1}\left(\mathrm{e}^{-\langle u+sh,\xi\rangle}-\mathrm{e}^{-\langle u,\xi\rangle}\right)\mathrm{d}s\right|\nonumber \\
 & \le\left|\langle h,\xi\rangle\right|^{2},\label{eq: esti for small xi}
\end{align}
where we used that $\langle u+sh,\xi\rangle\geq0$ and the Lipschitz continuity of $[0,\infty)\in x\mapsto\exp(-x)$
to get the last inequality. Similarly, for $\Vert\xi\Vert>M$,
\begin{equation}
\left|\mathrm{e}^{-\langle u,\xi\rangle}\left(1-\mathrm{e}^{-\langle h,\xi\rangle}-\langle h,\xi\rangle\right)\right|
\leq \left|\mathrm{e}^{-\langle u,\xi\rangle}-\mathrm{e}^{-\langle u+h,\xi\rangle}\right|+\left|\mathrm{e}^{-\langle u,\xi\rangle}\langle h,\xi\rangle\right|
\leq 2\left|\langle h,\xi\rangle\right|.\label{eq: esti for big xi}
\end{equation}
Combining \eqref{eq: esti for small xi}, \eqref{eq: esti for big xi}
and applying Lemma \ref{lem: triangle ineq.}, we get
\begin{align*}
\frac{1}{\Vert h\Vert} & \left\Vert \int_{\mathbb{S}_{d}^{+}\backslash\lbrace0\rbrace}\mathrm{e}^{-\langle u,\xi\rangle}\left(1-\mathrm{e}^{-\langle h,\xi\rangle}-\langle h,\xi\rangle\right)\mu(\mathrm{d}\xi)\right\Vert \\
 & \quad\leq\frac{1}{\Vert h\Vert}\int_{\mathbb{S}_{d}^{+}\backslash\lbrace0\rbrace}\left\vert \mathrm{e}^{-\langle u,\xi\rangle}\left(1-\mathrm{e}^{-\langle h,\xi\rangle}-\langle h,\xi\rangle\right)\right\vert \mathrm{tr}(\mu)(\mathrm{d}\xi)\\
 & \quad\leq\Vert h\Vert\int_{\lbrace\Vert\xi\Vert\leq M\rbrace}\Vert\xi\Vert^{2}\mathrm{tr}(\mu)(\mathrm{d}\xi)+2\int_{\lbrace\Vert\xi\Vert>M\rbrace}\Vert\xi\Vert\mathrm{tr}(\mu)(\mathrm{d}\xi),
\end{align*}
So
\[
\frac{\Vert r(u,h)\Vert}{\Vert h\Vert}\leq\left(2\Vert\alpha\Vert+\int_{\lbrace\Vert\xi\Vert\leq M\rbrace}\Vert\xi\Vert^{2}\mathrm{tr}(\mu)(\mathrm{d}\xi)\right)\Vert h\Vert+2\int_{\lbrace\Vert\xi\Vert>M\rbrace}\Vert\xi\Vert\mathrm{tr}(\mu)(\mathrm{d}\xi).
\]
Note that $\int_{\mathbb{S}_{d}^{+}\backslash\lbrace0\rbrace}\Vert\xi\Vert\mathrm{tr}(\mu)(\mathrm{d}\xi)<\infty$
by virtue of Definition \ref{def:admissible parameters} (iv). Let
$\varepsilon>0$ be arbitrary and fix some $M=M(\varepsilon)>0$ large
enough so that $\int_{\lbrace\Vert\xi\Vert>M\rbrace}\Vert\xi\Vert\mathrm{tr}(\mu)(\mathrm{d}\xi)<\varepsilon/4$.
Define
\[
\delta=\delta(\varepsilon):=\left(1+2\Vert\alpha\Vert+\int_{\lbrace\Vert\xi\Vert\leq M\rbrace}\Vert\xi\Vert^{2}\mathrm{tr}(\mu)(\mathrm{d}\xi)\right)^{-1}\frac{\varepsilon}{2}.
\]
Then, for $\Vert h\Vert\leq\delta$, we see that
\[
\frac{\Vert r(u,h)\Vert}{\Vert h\Vert}\leq\left(2\Vert\alpha\Vert+\int_{\lbrace\Vert\xi\Vert\leq M\rbrace}\Vert\xi\Vert^{2}\mathrm{tr}(\mu)(\mathrm{d}\xi)\right)\delta+\frac{\varepsilon}{2}\leq\varepsilon.
\]
This proves \eqref{eq: DR} for $u\in\mathbb{S}_{d}^{++}$. Finally,
the continuity of $u\mapsto$ $DR(u)$ in $\mathbb{S}_{d}^{+}$ can
be easily obtained from the dominated convergence theorem.

(b) Similarly as before, we derive $F(u+h)-F(u)=DF(u)(h)+r(u,h)$
with $r(u,h):=\int_{\mathbb{S}_{d}^{+}\backslash\lbrace0\rbrace}\exp(-\langle u,\xi\rangle)(1-\exp(\langle h,\xi\rangle)-\langle h,\xi\rangle)m(\mathrm{d}\xi)$.
Let $\Vert h\Vert\not=0$. By essentially the same reasoning as in
(a), we obtain that
\[
\frac{\Vert r(u,h)\Vert}{\Vert h\Vert}\leq\Vert h\Vert\int_{\lbrace\Vert\xi\Vert\leq M\rbrace}\Vert\xi\Vert^{2}m(\mathrm{d}\xi)+2\int_{\lbrace\Vert\xi\Vert>M\rbrace}\Vert\xi\Vert m(\mathrm{d}\xi),
\]
and the second integral on the right-hand side is now finite by \eqref{eq:first moment}.
Hence, we may follow the same steps as in (a) to see that $\Vert r(u,h)\Vert/\Vert h\Vert\to0$
as $\Vert h\Vert\to0$ and the continuity of $DF(u)$ in $\mathbb{S}_{d}^{+}$.
\end{proof}

Let $\phi$ and $\psi$ be as in Theorem \ref{thm:characterization theorem}.
We know from \cite[Lemma 3.2 (iii)]{MR2807963} that $\phi(t,u)$
and $\psi(t,u)$ are jointly continuous on $\mathbb{R}_{\geq0}\times\mathbb{S}_{d}^{+}$
and, moreover, $u\mapsto\phi(t,u)$ and $u\mapsto\psi(t,u)$ are analytic
on $\mathbb{S}_{d}^{++}$ for $t\geq0$.

\begin{prop}\label{prop:space-differentiability of phi and psi}
The following statements hold:
\begin{enumerate}
\item[(a)] $D\psi$ has a jointly continuous extension on $\mathbb{R}_{\geq0}\times\mathbb{S}_{d}^{+}$.
\item[(b)] If \eqref{eq:first moment} is satisfied, then $D\phi$ has a jointly
continuous extension on $\mathbb{R}_{\geq0}\times\mathbb{S}_{d}^{+}$.
\end{enumerate}
\end{prop}

\begin{proof} (a) Noting that $s\mapsto DR(\psi(s,u))\in L(\mathbb{S}_{d},\mathbb{S}_{d})$
is continuous, we may define $f_{u}(t)$ as the unique solution in
$L(\mathbb{S}_{d},\mathbb{S}_{d})$ to
\[
f_{u}(t)=\mathbbm{1}+\int_{0}^{t}DR\left(\psi(s,u)\right)f_{u}(s)\mathrm{d}s.
\]
Further, we then define the extension of $D\psi$ onto $\mathbb{R}_{\geq0}\times\partial\mathbb{S}_{d}^{+}$
simply by
\[
D\psi(t,u)=f_{u}(t),\quad(t,u)\in\mathbb{R}_{\geq0}\times\partial\mathbb{S}_{d}^{+}.
\]

It remains to verify the joint continuity of $D\psi(t,u)$ on $\mathbb{R}_{\geq0}\times\mathbb{S}_{d}^{+}$
extended in this way. By the Riccati differential equation \eqref{eq:riccati equation for psi}
we have
\[
D\psi(t,u)=\mathbbm{1}+\int_{0}^{t}DR\left(\psi(s,u)\right)D\psi(s,u)\mathrm{d}s,\quad t\geq0,\thinspace u\in\mathbb{S}_{d}^{+}.
\]
Using that $u\mapsto R(u)$ is continuous on $\mathbb{S}_{d}^{+}$
and $\psi$ is jointly continuous on $\mathbb{R}_{\geq0}\times\mathbb{S}_{d}^{+}$,
for all $T>0$ and $M>0$, there exists a constant $C(T,M)>0$ such
that
\[
\sup_{s\in[0,T],\thinspace u\in\mathbb{S}_{d}^{+},\thinspace\Vert u\Vert\leq M}\left\Vert DR\left(\psi(s,u)\right)\right\Vert =:C(T,M)<\infty.
\]
Hence, for each $u\in\mathbb{S}_{d}^{+}$ with $\Vert u\Vert\leq M$,
we obtain
\[
\left\Vert D\psi(t,u)\right\Vert \leq1+C(T,M)\int_{0}^{t}\Vert D\psi(s,u)\Vert\mathrm{d}s.
\]
Applying Gronwall's inequality yields
\[
\left\Vert D\psi(t,u)\right\Vert \leq\mathrm{e}^{C(T,M)T}=:K(T,M)<\infty,
\]
for all $t\in[0,T]$ and $u\in\mathbb{S}_{d}^{+}$ with $\Vert u\Vert\leq M$.
Because $D\psi$ is jointly continuous in $\mathbb{R}_{\geq0}\times\mathbb{S}_{d}^{++}$,
it is enough to prove continuity at some fixed point $(t,u)\in\mathbb{R}_{\geq0}\times\partial\mathbb{S}_{d}^{+}$,
where $\partial\mathbb{S}_{d}^{+}:=\mathbb{S}_{d}^{+}\backslash\mathbb{S}_{d}^{+}$.

Without loss of generality we assume $t\in[0,T]$ and $u\in\partial\mathbb{S}_{d}^{+}$
with $\|u\|\le M$. Let $s\in\mathbb{R}_{\geq0}$ and $v\in\mathbb{S}_{d}^{+}$
with $s\in[0,T]$ and $\Vert v\Vert\leq M$. We have
\begin{equation}
\left\Vert D\psi(t,u)-D\psi(s,v)\right\Vert \leq\left\Vert D\psi(t,u)-D\psi(s,u)\right\Vert +\left\Vert D\psi(s,u)-D\psi(s,v)\right\Vert .\label{eq:joint continuity 01}
\end{equation}
We estimate the first term on the right-hand side of \eqref{eq:joint continuity 01}
by
\begin{align}
\left\Vert D\psi(t,u)-D\psi(s,u)\right\Vert  & \leq\left\Vert \int_{0}^{t}DR\left(\psi(r,u)\right)D\psi(r,u)\mathrm{d}r-\int_{0}^{s}DR\left(\psi(r,u)\right)D\psi(r,u)\mathrm{d}r\right\Vert \nonumber \\
 & \leq C(T,M)\int_{[s,t]\cup[t,s]}\left\Vert D\psi(r,u)\right\Vert \mathrm{d}r\nonumber \\
 & \leq C(T,M)K(T,M)\vert t-s\vert.\label{eq:joint continuity 02}
\end{align}
Turning to the second term, for $v\in\mathbb{S}_{d}^{++}$ with $\Vert v\Vert\leq M$,
$D\psi(s,u)=f_{u}(s)$, and $D\psi(r,u)=f_{u}(r)$, we obtain
\begin{align*}
\left\Vert D\psi(s,u)-D\psi(s,v)\right\Vert  & \leq\int_{0}^{s}\left\Vert DR\left(\psi(r,u)\right)D\psi(r,u)-DR\left(\psi(r,v)\right)D\psi(r,v)\right\Vert \mathrm{d}r\\
 & \leq\int_{0}^{s}\left\Vert DR(\psi(r,u))-DR\left(\psi(r,v)\right)\right\Vert \left\Vert D\psi(r,v)\right\Vert \mathrm{d}r\\
 & \quad+\int_{0}^{s}\left\Vert DR\left(\psi(r,u)\right)\right\Vert \left\Vert D\psi(r,u)-D\psi(r,v)\right\Vert \mathrm{d}r\\
 & \leq K(T,M)\int_{0}^{T}\left\Vert DR(\psi(r,u))-DR\left(\psi(r,v)\right)\right\Vert \mathrm{d}r\\
 & \quad+C(T,M)\int_{0}^{s}\left\Vert D\psi(r,u)-D\psi(r,v)\right\Vert \mathrm{d}r\\
 & =K(T,M)a_{T}(v,u)+C(T,M)\int_{0}^{s}\left\Vert D\psi(r,u)-D\psi(r,v)\right\Vert \mathrm{d}r,
\end{align*}
where $a_{T}(v,u):=\int_{0}^{T}\Vert DR(\psi(r,u))-DR(\psi(r,v))\Vert\mathrm{d}r$.
Using once again Gronwall's inequality, we deduce
\begin{equation}
\left\Vert D\psi(s,u)-D\psi(s,v)\right\Vert \leq K(T,M)a_{T}(v,u)\mathrm{e}^{C(T,M)T}.\label{eq:joint continuity 03}
\end{equation}
Noting that $R\in C^{1}(\mathbb{S}_{d}^{+})$ and $\psi(r,0)=0$ by
\cite[Remark 2.5]{MR2807963}, by dominated convergence theorem, we
see that $a_{T}(v,u)$ tends to zero as $v\to u$. Consequently, the
right-hand side of \eqref{eq:joint continuity 03} tends to zero as
$v\to u$. Combining \eqref{eq:joint continuity 01} with \eqref{eq:joint continuity 02}
and \eqref{eq:joint continuity 03}, we conclude that $D\psi$ extended
in this way is jointly continuous in $(t,u)\in\mathbb{R}_{\geq0}\times\mathbb{S}_{d}^{+}$.

(b) We know from the generalized Riccati equation \eqref{eq:riccati equation for phi}
that $\phi(t,u)=\int_{0}^{t}F(\psi(s,u))\mathrm{d}s$. Noting that
$F\in C^{1}(\mathbb{S}_{d}^{+})$ due to \eqref{eq:first moment},
the chain rule combined with the dominated convergence theorem implies
the assertion. \end{proof}

We are ready to prove Theorem \ref{thm: first moment}.

\begin{proof}[Proof of Theorem \ref{thm: first moment}] Let
$\varepsilon>0$. We have
\begin{align*}
\frac{\partial}{\partial\varepsilon}\int_{\mathbb{S}_{d}^{+}}\mathrm{e}^{-\langle\varepsilon u,\xi\rangle}p_{t}(x,\mathrm{d}\xi) & =\frac{\partial}{\partial\varepsilon}\mathrm{e}^{-\phi(t,\varepsilon u)-\langle x,\psi(t,\varepsilon u)\rangle}\\
 & =-\left(D\phi(t,\varepsilon u)(u)+\langle x,D\psi(t,\varepsilon u)(u)\rangle\right)\mathrm{e}^{-\phi(t,\varepsilon u)-\langle x,\psi(t,\varepsilon u)\rangle}\\
 & \to-\left(D\phi(t,0)(u)+\langle x,D\psi(t,0)(u)\rangle\right)\quad\text{as }\varepsilon\to0,
\end{align*}
where we used that the functions $D\phi$ and $D\psi$ have a jointly
continuous extension on $\mathbb{R}_{\geq0}\times\mathbb{S}_{d}^{+}$
in accordance with Proposition \ref{prop:space-differentiability of phi and psi}.
On the other hand, noting $\left|\langle u,\xi\rangle\exp(-\langle\varepsilon u,\xi\rangle)\right|\le\varepsilon^{-1}\mathrm{e}^{-1}$
and applying dominated convergence theorem, we get
\[
\frac{\partial}{\partial\varepsilon}\int_{\mathbb{S}_{d}^{+}}\mathrm{e}^{-\langle\varepsilon u,\xi\rangle}p_{t}(x,\mathrm{d}\xi)=-\int_{\mathbb{S}_{d}^{+}}\langle u,\xi\rangle\mathrm{e}^{-\langle\varepsilon u,\xi\rangle}p_{t}(x,\mathrm{d}\xi)\to-\int_{\mathbb{S}_{d}^{+}}\langle u,\xi\rangle p_{t}(x,\mathrm{d}\xi)\quad\text{as }\varepsilon\to0.
\]
Note that the limit on the right-hand side is finite. Indeed, using
Fatou's lemma, we obtain
\[
\int_{\mathbb{S}_{d}^{+}}\langle u,\xi\rangle p_{t}(x,\mathrm{d}\xi)\leq\liminf_{\varepsilon\to0}\int_{\mathbb{S}_{d}^{+}}\langle u,\xi\rangle\mathrm{e}^{-\langle\varepsilon u,\xi\rangle}p_{t}(x,\mathrm{d}\xi)=D\phi(t,0)(u)+\langle x,D\psi(t,0)(u)\rangle<\infty
\]
for all $u\in\mathbb{S}_{d}^{+}$. So
\begin{equation}
\int_{\mathbb{S}_{d}^{+}}\langle u,\xi\rangle p_{t}(x,\mathrm{d}\xi)=D\phi(t,0)(u)+\langle x,D\psi(t,0)(u)\rangle.\label{eq:first moment eq1}
\end{equation}

In what follows, we compute the derivatives $D\phi(t,0)$ and $D\psi(t,0)$
explicitly. By means of the generalized Riccati equation \eqref{eq:riccati equation for psi},
we have
\[
\psi(t,u)-u=\int_{0}^{t}R\left(\psi(s,u)\right)ds,\quad t\geq0,\thinspace u\in\mathbb{S}_{d}^{+}.
\]
According to Lemma \ref{lem:space-differentiability of F and R} and
Proposition \ref{prop:space-differentiability of phi and psi} we
are allowed to differentiate both sides of the latter equation with
respect to $u\in\mathbb{S}_{d}^{+}$ and evaluate at $u=0$, thus,
using the dominated convergence theorem,
\[
\left.D\psi(t,u)\right|_{u=0}-\mathrm{Id}=\int_{0}^{t}\left.DR\left(\psi(s,u)\right)D\psi(s,u)\right\vert _{u=0}\mathrm{d}s,\quad t\geq0,
\]
where $\mathrm{Id}$ denotes the identity map on $\mathbb{S}_{d}^{+}$.
From \cite[Lemma 3.2 (iii)]{MR2807963} we know that $\psi(t,u)$
is continuous in $\mathbb{R}_{\geq0}\times\mathbb{S}_{d}^{+}$ and
noting that $\psi(s,0)=0$ (see \cite[Remark 2.5]{MR2807963}), we
get
\[
D\psi(t,0)-\mathrm{Id}=\int_{0}^{t}DR\left(0\right)D\psi(s,0)\mathrm{d}s,\quad t\geq0.
\]
From this and the precise formula for $\phi(t,h)$ we deduce that
\[
D\psi(t,0)=\mathrm{e}^{tDR(0)}\quad\text{and}\quad\quad D\phi(t,0)=\int_{0}^{t}DF(0)\mathrm{e}^{sDR(0)}\mathrm{d}s.
\]
We use Lemma \ref{lem:space-differentiability of F and R} to get
that
\[
DR(0)(u)=\widetilde{B}^{\top}(u)\quad\text{and}\quad DF(0)(u)=\langle b+\int_{\mathbb{S}_{d}^{+}\backslash\lbrace0\rbrace}\xi m(\mathrm{d}\xi),u\rangle.
\]
Finally, combining this with \eqref{eq:first moment eq1} yields
\begin{align*}
\int_{\mathbb{S}_{d}^{+}}\langle u,\xi\rangle p_{t}(x,\mathrm{d}\xi) & =\int_{0}^{t}\left(DF(0)\right)\mathrm{e}^{sDR(0)}(u)\mathrm{d}s+\langle x,\mathrm{e}^{tDR(0)}(u)\rangle\\
 & =\int_{0}^{t}\langle\mathrm{e}^{s\widetilde{B}}\left(b+\int_{\mathbb{S}_{d}^{+}\backslash\lbrace0\rbrace}\xi m(\mathrm{d}\xi)\right),u\rangle\mathrm{d}s+\langle\mathrm{e}^{t\widetilde{B}}x,u\rangle.
\end{align*}
Since the equality holds for each $u\in\mathbb{S}_{d}^{+}$, the assertion
is proved. \end{proof}

\section{Estimates on $\psi(t,u)$}

\label{sec:estimates on psi}

We fix an admissible parameter set $(\alpha,b,B,m,\mu)$ and let $\psi$
be the unique solution to \eqref{eq:riccati equation for psi}. In
this section we study upper and lower bounds for $\psi$. Let us start
with an upper bound for $\psi(t,u)$.

\begin{prop}\label{prop:upper bound for psi} Let $\psi$ be the
unique solution to \eqref{eq:riccati equation for psi}. Then
\begin{equation}
\left\Vert \psi\left(t,u\right)\right\Vert \leq M\Vert u\Vert\mathrm{e}^{-t\delta},\quad t\geq0,\label{eq:exponential convergence of psi}
\end{equation}
where $M$ and $\delta$ are given by \eqref{subcritical}. \end{prop}

\begin{proof} The proof is divided into three steps.

\textit{Step 1}: Denote by $q_{t}(x,\mathrm{d}\xi)$ the unique transition
kernel of an affine process on $\mathbb{S}_{d}^{+}$ with admissible
parameters $(\alpha,b,B,m=0,\mu)$, that is, for each $u,\thinspace x\in\mathbb{S}_{d}^{+}$,
we have
\begin{align}
\int_{\mathbb{S}_{d}^{+}}\mathrm{e}^{-\langle u,\xi\rangle}q_{t}(x,\mathrm{d}\xi)=\exp\left(-\int_{0}^{t}\langle b,\psi(s,u)\rangle\mathrm{d}s-\langle x,\psi(t,u)\rangle\right),\quad t\geq0.\label{eq: q laplace transform}
\end{align}
Applying Jensen's inequality
to the convex function $t\mapsto\exp(-t)$ yields
\begin{align*}
\int_{\mathbb{S}_{d}^{+}}\mathrm{e}^{-\langle u,\xi\rangle}q_{t}(x,\mathrm{d}\xi) & \geq\exp\left(-\int_{\mathbb{S}_{d}^{+}}\langle u,\xi\rangle q_{t}(x,\mathrm{d}\xi)\right)\\
 & =\exp\left(-\int_{0}^{t}\langle\mathrm{e}^{s\widetilde{B}}b,u\rangle\mathrm{d}s-\langle\mathrm{e}^{t\widetilde{B}}x,u\rangle\right),
\end{align*}
where the last identity is a special case of Theorem \ref{thm: first moment}.
Using \eqref{eq: q laplace transform} we obtain
\begin{equation}
\langle x,\psi(t,u)\rangle+\int_{0}^{t}\langle b,\psi(t,u)\rangle\mathrm{d}s\leq\langle\mathrm{e}^{t\widetilde{B}}x,u\rangle+\int_{0}^{t}\langle\mathrm{e}^{s\widetilde{B}}b,u\rangle\mathrm{d}s,\quad\text{for all }u,\thinspace x\in\mathbb{S}_{d}^{+},\ t\ge 0.\label{eq:psi inequality 1}
\end{equation}

\textit{Step 2}: Let $\alpha \in \mathbb{S}_{d}^{+}$ be fixed.  We claim that \eqref{eq:psi inequality 1} holds
not only for $b\succeq(d-1)\alpha$ but also for any $b\in\mathbb{S}_{d}^{+}$.
Aiming for a contradiction, suppose that there exist $t_{0}>0$
and $\xi,\thinspace x_{0},\thinspace u_{0}\in\mathbb{S}_{d}^{+}$
such that
\[
I:=\langle x_{0},\psi(t_{0},u_{0})\rangle+\int_{0}^{t_{0}}\langle\xi,\psi(s,u_{0})\rangle\mathrm{d}s-\langle x_{0},\mathrm{e}^{t_{0}\widetilde{B}^{\top}}u_{0}\rangle-\int_{0}^{t_{0}}\langle\xi,\mathrm{e}^{s\widetilde{B}^{\top}}u_{0}\rangle\mathrm{d}s>0.
\]
We now take an arbitrary but fixed $b_{0}\succeq(d-1)\alpha$. Noting
that
\[
\Delta:=\int_{0}^{t_{0}}\langle b_{0},\psi(s,u_{0})\rangle\mathrm{d}s-\int_{0}^{t_{0}}\langle b_{0},\mathrm{e}^{s\widetilde{B}^{\top}}u_{0}\rangle\mathrm{d}s
\]
is finite, we find a constant $K>0$ large enough so that $KI+\Delta>0$, i.e.,
\begin{equation}
\langle Kx_{0},\psi(t_{0},u_{0})\rangle+\int_{0}^{t_{0}}\langle b_{0}+K\xi,\psi(s,u_{0})\rangle\mathrm{d}s>\langle Kx_{0},\mathrm{e}^{t_{0}\widetilde{B}^{\top}}u_{0}\rangle+\int_{0}^{t_{0}}\langle b_{0}+K\xi,\mathrm{e}^{s\widetilde{B}^{\top}}u_{0}\rangle\mathrm{d}s.\label{eq:psi inequlatity 2}
\end{equation}
Now, since $b_{0}+K\xi\succeq(d-1)\alpha$, we see that \eqref{eq:psi inequlatity 2}
contradicts \eqref{eq:psi inequality 1} if we chose $b=b_{0}+K\xi$,
$x=Kx_{0}$, $u=u_{0}$, and $t=t_{0}$. Hence \eqref{eq:psi inequality 1}
holds for all $b\in\mathbb{S}_{d}^{+}$.

\textit{Step 3}: According to Step 2, we are allowed to choose $b=0$
in \eqref{eq:psi inequality 1}, which implies
\[
\langle x,\psi(t,u)\rangle\leq\langle x,\mathrm{e}^{t\widetilde{B}^{\top}}u\rangle
\]
for all $t\geq0$ and $x,\thinspace u\in\mathbb{S}_{d}^{+}$. This
completes the proof. \end{proof}

We continue with a lower bound for $\psi(t,u)$.

\begin{prop}\label{prop:lower bound for psi} Let $\psi$ be the
unique solution to \eqref{eq:riccati equation for psi} and suppose
that $\alpha=0$ and \eqref{eq:00} is satisfied. Then, for each $u,\thinspace\xi\in\mathbb{S}_{d}^{+}$,
\begin{equation}
\langle\xi,\psi(t,u)\rangle\geq\mathrm{e}^{-Kt}\langle\xi,u\rangle,\quad t\geq0.\label{eq: lower bound psi}
\end{equation}
\end{prop}

\begin{proof} Fix $u\in\mathbb{S}_{d}^{+}$ and define $W_{t}(u):=\psi(t,u)-\exp(-Kt)u$.
Using that $\exp(-Kt)u=\psi(t,u)-W_{t}(u)$ we obtain
\begin{align*}
\frac{\partial W_{t}(u)}{\partial t}=R(\psi(t,u))+K\psi(t,u)-KW_{t}(u).
\end{align*}
Since $W_{0}(u)=0$, the latter implies
\[
W_{t}(u)=\int_{0}^{t}\mathrm{e}^{-K(t-s)}\left(K\psi(s,u)+R(\psi(s,u))\right)\mathrm{d}s.
\]
Fix $\xi\in\mathbb{S}_{d}^{+}$, then
\begin{align}
\langle\xi,W_{t}(u)\rangle=\int_{0}^{t}\mathrm{e}^{-K(t-s)}\left(K\langle\xi,\psi(s,u)\rangle+\langle\xi,R(\psi(s,u))\rangle\right)\mathrm{d}s.\label{eq:02}
\end{align}
In the following we estimate the integrand. For this, we write $\langle\xi,R(\psi(s,u))\rangle=I_{1}+I_{2}$,
where
\[
I_{1}=\langle\xi,B{}^{\top}\left(\psi(s,u)\right)\rangle\quad\text{and}\quad I_{2}=-\int_{\mathbb{S}_{d}^{+}\backslash\{0\}}\left(\mathrm{e}^{-\langle\psi(s,u),\zeta\rangle}-1\right)\langle\xi,\mu\left(\mathrm{d}\zeta\right)\rangle,
\]
and estimate $I_{1}$ and $I_{2}$ separately. For $I_{1}$, by
\eqref{eq:00} we get
\[
I_{1}=\langle B(\xi),\psi(s,u)\rangle\geq-K\langle\xi,\psi(s,u)\rangle,
\]
where we used the self-duality of the cone $\mathbb{S}_d^+$ (see \cite[Theorem 7.5.4]{MR2978290}).
Turning to $I_{2}$, we simply have
\begin{align*}
I_{2} & =\int_{\mathbb{S}_{d}^{+}\backslash\{0\}}\left(1-\mathrm{e}^{-\langle\psi(s,u),\zeta\rangle}\right)\langle\xi,\mu\left(\mathrm{d}\zeta\right)\rangle\geq0.
\end{align*}
Collecting now the estimates for $I_{1}$ and $I_{2}$, we see that
\[
\left(K\langle\xi,\psi(s,u)\rangle+\langle\xi,R(\psi(s,u))\rangle\right)\geq0
\]
and, thus, $\langle\xi,W_{t}(u)\rangle\geq0$ by \eqref{eq:02} .
This proves the assertion. \end{proof}

\section{Proof of the main results}

\label{sec:proof of ergodicity}

In this section we will prove Theorem \ref{thm:existence of limit distribution},
Proposition \ref{prop:necessity for ergodicity}, and Corollary \ref{cor:first moment of pi}.
Let $p_{t}(x,\mathrm{d}\xi)$ be the transition kernel of a subcritical
affine process on $\mathbb{S}_{d}^{+}$ with admissible parameters
$(\alpha,b,B,m,\mu)$ and $\delta>0$ be given by \eqref{subcritical}.

We note that $F(u)\geq0$ for all $u\in\mathbb{S}_{d}^{+}$. Based on
the estimates on $\psi(t,u)$ that we derived in the previous section,
we easily obtain the following lemma.

\begin{lem}\label{lem:bound for F(psi(s,u))} Suppose that \eqref{LOGMOMENT}
holds. Then there exists a constant $C>0$ such that
\begin{equation}
F(\psi(s,u))\leq C\Vert u\Vert\mathrm{e}^{-s\delta},\quad s\geq0,\ \ u\in\mathbb{S}_{d}^{+}.\label{eq:bound for F}
\end{equation}
Consequently,
\begin{align}
\int_{0}^{\infty}F(\psi(s,u))\mathrm{d}s\leq\frac{C}{\delta}\|u\|,\quad u\in\mathbb{S}_{d}^{+}.\label{eq:06}
\end{align}
\end{lem}

\begin{proof} We know that
\begin{align*}
F(\psi(s,u)) & =\langle b,\psi(s,u)\rangle+\int_{\mathbb{S}_{d}^{+}\backslash\lbrace0\rbrace}\left(1-\mathrm{e}^{-\langle\psi(s,u),\xi\rangle}\right)m(\mathrm{d}\xi)\\
 & =:\langle b,\psi(s,u)\rangle+I(u).
\end{align*}
Now, first note that, by \eqref{eq:exponential convergence of psi},
\begin{equation}
\langle b,\psi(s,u)\rangle\leq\Vert b\Vert\Vert\psi(s,u)\Vert\leq\Vert b\Vert\Vert u\Vert\mathrm{e}^{-s\delta}.\label{eq:bound for I_1}
\end{equation}
We turn to estimate $I(u)$. Using once again \eqref{eq:exponential convergence of psi},
we obtain
\begin{align*}
I(u) & =\int_{\mathbb{S}_{d}^{+}\backslash\lbrace0\rbrace}\left(1-\mathrm{e}^{-\langle\psi(s,u),\xi\rangle}\right)m(\mathrm{d}\xi)\\
 & \leq\int_{\mathbb{S}_{d}^{+}\backslash\lbrace0\rbrace}\min\left\lbrace 1,\langle\psi(s,u),\xi\rangle\right\rbrace m(\mathrm{d}\xi)\\
 & \leq\int_{\mathbb{S}_{d}^{+}\backslash\lbrace0\rbrace}\min\left\lbrace 1,\Vert\xi\Vert\Vert u\Vert\mathrm{e}^{-s\delta}\right\rbrace m(\mathrm{d}\xi).
\end{align*}
For all $a\geq0$ it holds $1\wedge a\leq\log(2)^{-1}\log(1+a)$,
hence
\begin{align*}
I(u) & \leq\frac{1}{\log(2)}\int_{\lbrace\Vert\xi\Vert\leq1\rbrace}\Vert\xi\Vert\Vert u\Vert\mathrm{e}^{-s\delta}m(\mathrm{d}\xi)+\frac{1}{\log(2)}\int_{\lbrace\Vert\xi\Vert>1\rbrace}\log\left(1+\Vert\xi\Vert\Vert u\Vert\mathrm{e}^{-s\delta}\right)m(\mathrm{d}\xi)\\
 & =:J_{1}(u)+J_{2}(u).
\end{align*}
Let $C>0$ be a generic constant which may vary from line to line.
Since $m(\mathrm{d}\xi)$ integrates $\Vert\xi\Vert\mathbbm{1}_{\lbrace\Vert\xi\Vert\leq1\rbrace}$
by definition, we have
\[
J_{1}(u)\leq C\Vert u\Vert\mathrm{e}^{-s\delta}.
\]
Moreover, noting that $m(\mathrm{d}\xi)$ integrates $\log\Vert\xi\Vert\mathbbm{1}_{\lbrace\Vert\xi\Vert>1\rbrace}$
by assumption, for $J_{2}(u)$ we use the elementary inequality (see \cite[Lemma 8.5]{2019arXiv190105815F})
\begin{align*}
\log(1+a\cdot c) & \leq C\min\left\lbrace \log(1+a),\log(1+c)\right\rbrace +C\log(1+a)\log(1+c)\\
 & \leq C\log(1+a)+Ca\log(1+c)\\
 & \leq Ca\left(1+\log(1+c)\right)
\end{align*}
for $a=\Vert u\Vert\exp(-s\delta)$ and $c=\Vert\xi\Vert$ to get
\[
J_{2}(u)\leq C\Vert u\Vert\mathrm{e}^{-s\delta}\int_{\lbrace\Vert\xi\Vert>1\rbrace}\left(1+\log\left(1+\Vert\xi\Vert\right)\right)m(\mathrm{d}\xi)\leq C\Vert u\Vert\mathrm{e}^{-s\delta}.
\]
Combining the estimates for $J_{1}(u)$ and $J_{2}(u)$ yields
\begin{equation}
I(u)=J_{1}(u)+J_{2}(u)\leq C\Vert u\Vert\mathrm{e}^{-s\delta}.\label{eq:bound for I_2}
\end{equation}
So, by \eqref{eq:bound for I_1} and \eqref{eq:bound for I_2}, we
have \eqref{eq:bound for F} which proves the assertion. \end{proof}

We are now able to prove Theorem \ref{thm:existence of limit distribution}.

\begin{proof}[Proof of Theorem \ref{thm:existence of limit distribution}]
Fix $x\in\mathbb{S}_{d}^{+}$. By means of Proposition \ref{prop:upper bound for psi},
we see that
\begin{align*}
\lim_{t\to\infty}\int_{\mathbb{S}_{d}^{+}}\mathrm{e}^{-\langle u,\xi\rangle}p_{t}(x,\mathrm{d}\xi) & =\lim_{t\to\infty}\exp\left(-\int_{0}^{t}F(\psi(s,u))\mathrm{d}s-\langle x,\psi(t,u)\rangle\right)\\
 & =\exp\left(-\int_{0}^{\infty}F(\psi(s,u))\mathrm{d}s\right),
\end{align*}
and the limit on the right-hand side is finite according to Lemma
\ref{lem:bound for F(psi(s,u))}. Clearly, by \eqref{eq:06}, we also
have that $u\mapsto\int_{0}^{\infty}F(\psi(s,u)))\mathrm{d}s$ is
continuous at $u=0$. Now, L\'evy's continuity theorem, cf. \cite[Lemma 4.5]{MR2807963},
implies that $p_{t}(x,\cdot)\rightarrow\pi$ weakly as $t\to\infty$.
Moreover, $\pi$ has Laplace transform \eqref{eq:laplace transform of pi}.
It remains to verify that $\pi$ is the unique invariant distribution.

\textit{Invariance.} Fix $u\in\mathbb{S}_{d}^{+}$ and let $t\geq0$
be arbitrary. Then
\begin{align*}
\int_{\mathbb{S}_{d}^{+}}\mathrm{e}^{-\langle u,\xi\rangle}\left(\int_{\mathbb{S}_{d}^{+}}p_{t}(x,\mathrm{d}\xi)\pi(\mathrm{d}x)\right) & =\int_{\mathbb{S}_{d}^{+}}\left(\int_{\mathbb{S}_{d}^{+}}\mathrm{e}^{-\langle u,\xi\rangle}p_{t}(x,\mathrm{d}\xi)\right)\pi(\mathrm{d}x)\\
 & =\mathrm{e}^{-\phi(t,u)}\int_{\mathbb{S}_{d}^{+}}\exp\left(-\langle x,\psi(t,u)\rangle\right)\pi(\mathrm{d}x).
\end{align*}
Noting that $\psi$ satisfies the semi-flow equation\footnote{I.e., it holds that $\psi(t+s,u)=\psi\left(s,\psi(t,u)\right)$ for
all $t,\thinspace s\geq0$.} due to \cite[Lemma 3.2]{MR2807963} and using that the Laplace transform
of $\pi$ is given by \eqref{eq:laplace transform of pi}, for each
$u\in\mathbb{S}_{d}^{+}$, we obtain
\begin{align*}
\int_{\mathbb{S}_{d}^{+}}\mathrm{e}^{-\langle u,\xi\rangle}\left(\int_{\mathbb{S}_{d}^{+}}p_{t}(x,\mathrm{d}\xi)\pi(\mathrm{d}x)\right) & =\mathrm{e}^{-\phi(t,u)}\exp\left(-\int_{0}^{\infty}F\left(\psi\left(s,\psi(t,u)\right)\right)\mathrm{d}s\right)\\
 & =\mathrm{e}^{-\phi(t,u)}\exp\left(-\int_{0}^{\infty}F\left(\psi(t+s,u)\right)\mathrm{d}s\right)\\
 & =\mathrm{e}^{-\phi(t,u)}\exp\left(-\int_{t}^{\infty}F\left(\psi(s,u)\right)\mathrm{ds}\right)\\
 & =\exp\left(-\int_{0}^{\infty}F\left(\psi(s,u)\right)\mathrm{d}s\right)\\
 & =\int_{\mathbb{S}_{d}^{+}}\mathrm{e}^{-\langle x,u\rangle}\pi\left(\mathrm{d}x\right).
\end{align*}
Consequently, $\pi$ is invariant.

\textit{Uniqueness.} Let $\pi'$ be another invariant distribution.
For fixed $u\in\mathbb{S}_{d}^{+}$ and $t\geq0$ we have
\begin{align*}
\int_{\mathbb{S}_{d}^{+}}\mathrm{e}^{-\langle x,u\rangle}\pi'(\mathrm{d}x) & =\int_{\mathbb{S}_{d}^{+}}\mathrm{e}^{-\langle u,\xi\rangle}\left(\int_{\mathbb{S}_{d}^{+}}p_{t}(x,\mathrm{d}\xi)\pi'(\mathrm{d}x)\right)\\
 & =\int_{\mathbb{S}_{d}^{+}}\exp\left(-\phi(t,u)-\langle x,\psi(t,u)\rangle\right)\pi'(\mathrm{d}x).
\end{align*}
Letting $t\to\infty$ shows that $\pi'$ also satisfies \eqref{eq:laplace transform of pi}.
By uniqueness of the Laplace transforms, it holds that $\pi'=\pi$.
\end{proof}

\begin{proof}[Proof of Proposition \ref{prop:necessity for ergodicity}]
Let $x\in\mathbb{S}_{d}^{+}$ and $\pi\in\mathcal{P}(\mathbb{S}_{d}^{+})$
be such that $p_{t}(x,\cdot)\to\pi$ weakly as $t\to\infty$. It follows
that
\[
\lim_{t\to\infty}\int_{\mathbb{S}_{d}^{+}}\mathrm{e}^{-\langle u,\xi\rangle}p_{t}(x,\mathrm{d}\xi)=\int_{\mathbb{S}_{d}^{+}}\mathrm{e}^{-\langle u,y\rangle}\pi(\mathrm{d}y),\quad u\in\mathbb{S}_{d}^{+},
\]
and we obtain from \eqref{eq:affine representation}
\[
\lim_{t\to\infty}\exp\left(-\int_{0}^{t}F(\psi(s,u))\mathrm{d}s\right)=\lim_{t\to\infty}\mathrm{e}^{\langle x,\psi(t,u)\rangle}\int_{\mathbb{S}_{d}^{+}}\mathrm{e}^{-\langle u,\xi\rangle}p_{t}(x,\mathrm{d}\xi)=\int_{\mathbb{S}_{d}^{+}}\mathrm{e}^{-\langle u,y\rangle}\pi(\mathrm{d}y).
\]
In particular, this implies
\begin{align*}
\int_{0}^{\infty}F(\psi(s,u))\mathrm{d}s<\infty,\quad u\in\mathbb{S}_{d}^{+}.
\end{align*}
Fix $u\in\mathbb{S}_{d}^{++}$. Assume that $\alpha=0$ and \eqref{eq:00}
holds. By definition of $F$ we have $F(u)\geq\int_{\mathbb{S}_{d}^{+}}(1-\exp(-\langle u,\xi\rangle))m(\mathrm{d}\xi)$
and thereby
\begin{align*}
F(\psi(s,u))\geq\int_{\lbrace\langle\xi,u\rangle>1\rbrace}\left(1-\mathrm{e}^{-\mathrm{e}^{-Ks}\langle\xi,u\rangle}\right)m(\mathrm{d}\xi),
\end{align*}
where we used \eqref{eq: lower bound psi}. Integrating over $[0,\infty)$
and using a change of variable $r:=\exp(-Ks)\langle\xi,u\rangle$
with $\mathrm{d}s=-1/K\cdot\mathrm{d}r/r$ yields
\begin{align*}
\int_{0}^{\infty}F(\psi(s,u))\mathrm{d}s & \geq\frac{1}{K}\int_{\lbrace\langle\xi,u\rangle>1\rbrace}\int_{0}^{\langle\xi,u\rangle}\frac{1-\mathrm{e}^{-r}}{r}\mathrm{d}rm(\mathrm{d}\xi)\\
 & \geq\frac{1}{K}\int_{\lbrace\langle\xi,u\rangle>1\rbrace}\int_{1}^{\langle\xi,u\rangle}\frac{1-\mathrm{e}^{-r}}{r}\mathrm{d}rm(\mathrm{d}\xi)\\
 & \geq\frac{1-\mathrm{e}^{-1}}{K}\int_{\lbrace\langle\xi,u\rangle>1\rbrace}\log(\langle\xi,u\rangle)m(\mathrm{d}\xi),
\end{align*}
where we used in the last inequality that $1-\exp(-r)\geq1-\exp(-1)>0$
for $r\ge1$. This leads to the estimate
\[
\int_{\lbrace\langle\xi,u\rangle>1\rbrace}\log(\langle\xi,u\rangle)m(\mathrm{d}\xi)\leq\frac{K}{1-\mathrm{e}^{-1}}\int_{0}^{\infty}F(\psi(s,u))\mathrm{d}s<\infty.
\]
Letting $u=\mathbbm{1}\in\mathbb{S}_{d}^{++}$ gives $\langle\xi,\mathbbm{1}\rangle=\mathrm{tr}(\xi)\geq\Vert\xi\Vert$
so that
\[
\int_{\lbrace\Vert\xi\Vert>1\rbrace}\log\left(\Vert\xi\Vert\right)m(\mathrm{d}\xi)\leq\int_{\lbrace\langle\xi,\mathbbm{1}\rangle>1\rbrace}\log\left(\langle\xi,\mathbbm{1}\rangle\right)m(\mathrm{d}\xi)<\infty.
\]
This completes the proof. \end{proof}

\begin{proof}[Proof of Corollary \ref{cor:first moment of pi}]
Using that $\Vert\exp(t\widetilde{B})\Vert\leq M\exp(-\delta t)$,
where $\delta$ is given by \eqref{subcritical}, we have
\begin{align*}
\lim_{t\to\infty}\int_{\mathbb{S}_{d}^{+}}yp_{t}(x,\mathrm{d}y) & =\int_{0}^{\infty}\mathrm{e}^{s\widetilde{B}}\left(b+\int_{\mathbb{S}_{d}^{+}}\xi m(\mathrm{d}\xi)\right)\mathrm{d}s\in\mathbb{S}_{d}^{+}.
\end{align*}
It remains to verify that $\lim_{t\to\infty}\int_{\mathbb{S}_{d}^{+}}yp_{t}(x,\mathrm{d}y)=\int_{\mathbb{S}_{d}^{+}}y\pi(\mathrm{d}y)$.
To do so, we can proceed similar to the proof of Theorem \ref{thm: first moment}.
Indeed, by Lemma \ref{lem:estimate0}, we estimate
\[
\sup_{t\geq0}\int_{\mathbb{S}_{d}^{+}}\Vert y\Vert p_{t}(x,\mathrm{d}y)\leq\sup_{t\geq0}\mathrm{tr}\left(\int_{\mathbb{S}_{d}^{+}}yp_{t}(x,\mathrm{d}y)\right)\leq\sqrt{d}\sup_{t\geq0}\left\Vert \int_{\mathbb{S}_{d}^{+}}yp_{t}(x,\mathrm{d}y)\right\Vert <\infty.
\]
Therefore, applying the Lemma of Fatou yields
\[
\int_{\mathbb{S}_{d}^{+}}\Vert y\Vert\pi(\mathrm{d}y)\leq\sup_{t\geq0}\int_{\mathbb{S}_{d}^{+}}\Vert y\Vert p_{t}(x,\mathrm{d}y)<\infty.
\]
So $\pi\in\mathcal{P}_{1}(\mathbb{S}_{d}^{+})$. Now, let $\varepsilon>0$.
By dominated convergence theorem, we see that
\[
\lim_{\varepsilon\searrow0}\int_{\mathbb{S}_{d}^{+}}\frac{1-\mathrm{e}^{-\langle\varepsilon u,y\rangle}}{\varepsilon}\pi(\mathrm{d}y)=\int_{\mathbb{S}_{d}^{+}}\langle u,y\rangle\pi(\mathrm{d}y).
\]
Moreover, Noting that, by Proposition \ref{prop:upper bound for psi},
\[
1-\mathrm{e}^{-\langle\psi(s,\varepsilon u),\xi\rangle}\leq\langle\psi(s,\varepsilon u),\xi\rangle\leq\Vert\xi\Vert\Vert\varepsilon u\Vert\mathrm{e}^{-\delta t},
\]
we can use once again the dominated convergence theorem to obtain
\begin{align*}
\lim_{\varepsilon\searrow0}\int_{\mathbb{S}_{d}^{+}}\frac{1-\mathrm{e}^{-\langle\varepsilon u,y\rangle}}{\varepsilon}\pi(\mathrm{d}y) & =\lim_{\varepsilon\searrow0}\frac{1}{\varepsilon}\int_{0}^{\infty}F(\psi(s,\varepsilon u))\mathrm{d}s\\
 & =\lim_{\varepsilon\searrow0}\int_{0}^{\infty}\left(\langle b,\frac{\psi(s,\varepsilon u)}{\varepsilon}\rangle\mathrm{d}s+\int_{\mathbb{S}_{d}^{+}\backslash\lbrace0\rbrace}\frac{1-\mathrm{e}^{-\langle\psi(s,\varepsilon u),\xi\rangle}}{\varepsilon}m(\mathrm{d}\xi)\right)\mathrm{d}s\\
 & =\int_{0}^{\infty}\langle b,D\psi(s,0)(u)\rangle\mathrm{d}s+\int_{0}^{\infty}\int_{\mathbb{S}_{d}^{+}\backslash\lbrace0\rbrace}\langle D\psi(s,0)(u),\xi\rangle m(\mathrm{d}\xi)\mathrm{d}s\\
 & =\int_{0}^{\infty}\langle\mathrm{e}^{s\widetilde{B}}\left(b+\int_{\mathbb{S}_{d}^{+}\backslash\lbrace0\rbrace}\xi m(\mathrm{d}\xi)\right),u\rangle\mathrm{d}s,
\end{align*}
where we used that $D\psi(s,0)(u)=\exp(s\widetilde{B}^{\top})u$ (see
the proof of Theorem \ref{thm: first moment}). Since the latter identity
holds for all $u\in\mathbb{S}_{d}^{+}$, we conclude with our proof.
\end{proof}

\section{Proof of Theorem \ref{thm:convergence dL}}

\label{sec:proof of convergence in dL}

\begin{proof}[Proof of Theorem \ref{thm:convergence dL}]
Suppose that \eqref{LOGMOMENT} holds. By definition of $d_{L}$, we have
\begin{align}
 & d_{L}\left(p_{t}(x,\mathrm{d}\xi),\pi(\mathrm{d}\xi)\right)\nonumber \\
 & =\sup_{u\in\mathbb{S}_{d}^{+}\setminus\{0\}}\frac{1}{\Vert u\Vert}\left\vert \int_{\mathbb{S}_{d}^{+}}\mathrm{e}^{-\langle u,\xi\rangle}p_{t}(x,\mathrm{d}\xi)-\int_{\mathbb{S}_{d}^{+}}\mathrm{e}^{-\langle u,\xi\rangle}\pi(\mathrm{d}\xi)\right\vert \nonumber \\
 & =\sup_{u\in\mathbb{S}_{d}^{+}\setminus\{0\}}\frac{1}{\Vert u\Vert}\left\vert \exp\left(-\int_{0}^{t}F\left(\psi(s,u)\right)\mathrm{d}s-\langle x,\psi(t,u)\rangle\right)-\exp\left(-\int_{0}^{\infty}F\left(\psi(s,u)\right)\mathrm{d}s\right)\right\vert .\label{eq:01}
\end{align}
Let $C>0$ be a generic constant that may vary from line to line.
Using then \eqref{eq:bound for F}, we have, for each $t\geq0$,
\begin{align*}
 & \left\vert \exp\left(-\int_{0}^{t}F\left(\psi(s,u)\right)\mathrm{d}s-\langle x,\psi(t,u)\rangle\right)-\exp\left(-\int_{0}^{\infty}F\left(\psi(s,u)\right)\mathrm{d}s\right)\right\vert \\
 & \qquad\qquad\qquad\leq\left\vert \exp\left(-\langle x\psi(t,u)\rangle\right)-1\right\vert \cdot\left|\exp\left(-\int_{0}^{t}F\left(\psi(s,u)\right)\mathrm{d}s\right)\right\vert \\
 & \qquad\qquad\qquad\quad+\left\vert \exp\left(-\int_{0}^{t}F\left(\psi(s,u)\right)\mathrm{d}s\right)-\exp\left(-\int_{0}^{\infty}F\left(\psi(s,u)\right)\mathrm{d}s\right)\right\vert \\
 & \qquad\qquad\qquad\leq\left\vert \langle x,\psi(s,u)\rangle\right|+\left|\int_{t}^{\infty}F\left(\psi(s,u)\right)\mathrm{d}s\right\vert \\
 & \qquad\qquad\qquad\leq M\Vert x\Vert\Vert u\Vert\mathrm{e}^{-t\delta}+C\Vert u\Vert\int_{t}^{\infty}\mathrm{e}^{-s\delta}\mathrm{d}s\\
 & \qquad\qquad\qquad\leq C\left(1+\Vert x\Vert\right)\Vert u\Vert\mathrm{e}^{-t\delta},
\end{align*}
which when plugged back into \eqref{eq:01} implies \eqref{eq:convergence speed of pi}.
\end{proof}

\section{Proof of Theorem \ref{th:wasserstein ergodicity}}

\label{sec:proof of exp ergodicity in wasserstein distance}

\begin{proof}[Proof of Theorem \ref{th:wasserstein ergodicity}]
Note that $\pi\in\mathcal{P}_{1}(\mathbb{S}_{d}^{+})$ by Corollary
\ref{cor:first moment of pi}. Let $q_{t}(x,\mathrm{d}\xi)$ be transition
kernel for the conservative, subcritical affine processes with admissible
parameters $(\alpha=0,b=0,B,m=0,\mu)$. Using the particular form
of the Laplace transform for $p_{t}(x,\cdot)$ (see \eqref{eq:affine representation})
it is not difficult to see that $p_{t}(x,\cdot)=q_{t}(x,\cdot)\ast p_{t}(0,\cdot)$,
where `$\ast$' denotes the convolution of measures. Let $H$ be any
coupling with marginals $\delta_{x}$ and $\pi$, i.e., $H\in\mathcal{H}(\delta_{x},\pi)$.
Using the invariance of $\pi$, together with the convexity of $W_{1}$
(see \cite[Theorem 4.8]{MR2459454}) and \cite[Lemma 2.3]{2019arXiv190202833F},
we find
\begin{align*}
W_{1}\left(p_{t}(x,\cdot),\pi\right) & =W_{1}\left(\int_{\mathbb{S}_{d}^{+}}p_{t}(y,\cdot)\delta_{x}(\mathrm{d}y),\int_{\mathbb{S}_{d}^{+}}p_{t}(y',\cdot)\pi(\mathrm{d}y')\right)\\
 & \leq\int_{\mathbb{S}_{d}^{+}\times\mathbb{S}_{d}^{+}}W_{1}\left(p_{t}(y,\cdot),p_{t}(y',\cdot)\right)H(\mathrm{d}y,\mathrm{d}y')\\
 & \leq\int_{\mathbb{S}_{d}^{+}\times\mathbb{S}_{d}^{+}}W_{1}\left(q_{t}(y,\cdot),q_{t}(y',\cdot)\right)H(\mathrm{d}y,\mathrm{d}y').
\end{align*}
The integrand can now be estimated as follows
\begin{align*}
W_{1}\left(q_{t}(y,\cdot),q_{t}(y',\cdot)\right) & \leq\int_{\mathbb{S}_{d}^{+}\times\mathbb{S}_{d}^{+}}\|z-z'\|G(\mathrm{d}z,\mathrm{d}z')\\
 & \leq\int_{\mathbb{S}_{d}^{+}}\|z\|q_{t}(y,\mathrm{d}z)+\int_{\mathbb{S}_{d}^{+}}\|z'\|q_{t}(y',\mathrm{d}z')\\
 & \leq M\sqrt{d}\mathrm{e}^{-t\delta}\left(\|y\|+\|y'\|\right),
\end{align*}
where $G$ is any coupling of $(q_{t}(y,\cdot),q_{t}(y',\cdot))$
and we have used Lemma \ref{lem:estimate0} to obtain
\begin{align*}
\int_{\mathbb{S}_{d}^{+}}\Vert z\Vert q_{t}(y,\mathrm{d}z) & \leq\mathrm{tr}\left(\int_{\mathbb{S}_{d}^{+}}zq_{t}(y,\mathrm{d}z)\right)\\
 & =\mathrm{tr}\left(\mathrm{e}^{t\widetilde{B}}y\right)\\
 & \leq\sqrt{d}\left\Vert \mathrm{e}^{t\widetilde{B}}y\right\Vert \\
 & \leq M\sqrt{d}\mathrm{e}^{-t\delta}\Vert y\Vert.
\end{align*}
Combining these estimates, we obtain
\begin{align*}
W_{1}(p_{t}(x,\cdot),\pi) & \leq M\sqrt{d}\mathrm{e}^{-t\delta}\int_{\mathbb{S}_{d}^{+}\times\mathbb{S}_{d}^{+}}\left(\|y\|+\|y'\|\right)H(\mathrm{d}y,\mathrm{d}y')\\
 & \leq M\sqrt{d}\mathrm{e}^{-t\delta}\left(\|x\|+\int_{\mathbb{S}_{d}^{+}}\Vert y\Vert\pi(\mathrm{d}y)\right),
\end{align*}
which yields \eqref{eq:05}. \end{proof}

\section{Applications }

\label{sec:applications}

Let $(W_{t})_{t\geq0}$ be a $d\times d$-matrix of independent standard
Brownian motions. Denote by $(J_{t})_{t\geq0}$ an $\mathbb{S}_{d}^{+}$-valued
L\'evy subordinator with L\'evy measure $m$. Suppose that these two processes
are independent of each other. Following \cite{MR2819242}, the stochastic
differential equation
\begin{equation}
\begin{cases}
\mathrm{d}X_{t}=\left(b+\beta X_{t}+X_{t}\beta^{\top})\right)\mathrm{d}t+\sqrt{X_{t}}\mathrm{d}W_{t}\Sigma+\Sigma^{\top}\mathrm{d}W_{t}^{\top}\sqrt{X_{t}}+\mathrm{d}J_{t} & t\geq0,\\
X_{0}=x\in\mathbb{S}_{d}^{+},
\end{cases}\label{eq:sde affine jump diffusion}
\end{equation}
has a unique weak solution if $b\succeq(d-1)\Sigma^{\top}\Sigma$
and $\Sigma,\thinspace\beta$ are real-valued $d\times d$-matrices.
Moreover, according to \cite[Corollary 3.2]{MR2819242}, if $b\succ(d+1)\Sigma^{\top}\Sigma$,
then a unique strong solution also exists. The corresponding Markov
process $X=(X_{t})_{t\geq0}$ is a conservative affine process with
admissible parameters $(\alpha,b,B,m,0)$ with diffusion $\alpha=\Sigma^{\top}\Sigma$
and linear drift $B(x)=\beta x+x\beta^{\top}$. The functions $F$
and $R$ are given by
\[
F(u)=\langle b,u\rangle+\int_{\mathbb{S}_{d}^{+}\backslash\lbrace0\rbrace}\left(1-\mathrm{e}^{-\langle u,\xi\rangle}\right)m(\mathrm{d}\xi)
\]
and
\[
R(u)=-2u\alpha u+u\beta+\beta^{\top}u.
\]
The generalized Riccati equations are now given by
\begin{align*}
\partial_{t}\phi(t,u) & =\langle b,\psi(t,u)\rangle+\int_{\mathbb{S}_{d}^{+}\backslash\{0\}}\left(1-\mathrm{e}^{-\langle\psi(t,u),\xi\rangle}\right)m(\mathrm{d}\xi),\\
\partial_{t}\psi(t,u) & =-2\psi(t,u)\alpha\psi(t,u)+\psi(t,u)\beta+\beta^{\top}\psi(t,u),
\end{align*}
with initial conditions $\phi(0,u)=0$ and $\psi(0,u)=u$. Let $\sigma_{t}^{\beta}:\mathbb{S}_{d}^{+}\to\mathbb{S}_{d}^{+}$
be given by
\[
\sigma_{t}^{\beta}(x):=2\int_{0}^{t}\mathrm{e}^{\beta s}x\mathrm{e}^{\beta^{\top}s}\mathrm{d}s,\quad t\geq0.
\]
According to \cite[Section 4.3]{MR2956112}, we have
\begin{align*}
\phi(t,u) & =\langle b,\int_{0}^{t}\psi(s,u)\mathrm{d}s \rangle +\int_{0}^{t}\int_{\mathbb{S}_{d}^{+}\backslash\{0\}}\left(1-\mathrm{e}^{-\langle\psi(s,u),\xi\rangle}\right)m(\mathrm{d}\xi)\mathrm{d}s,\\
\psi(t,u) & =\mathrm{e}^{\beta^{\top}t}\left(u^{-1}+\sigma_{t}^{\beta}(\alpha)\right)^{-1}\mathrm{e}^{\beta t}.
\end{align*}
Since $\widetilde{B}(x)=B(x)$, Remark \ref{rem:equivalence for delta condition}
implies that $X$ is subcritical, provided $\beta$ has only eigenvalues
with negative real parts. If the L\'evy measure $m$ satisfies \eqref{LOGMOMENT},
then Theorem \ref{thm:existence of limit distribution} implies existence,
uniqueness, and convergence to the invariant distribution $\pi$ whose
Laplace transform satisfies
\[
\int_{0}^{\infty}\mathrm{e}^{-\langle u,x\rangle}\pi\left(\mathrm{d}x\right)  =\langle b,\int_{0}^{\infty}\psi(s,u)\mathrm{d}s\rangle \exp\left(-\int_{0}^{\infty}\int_{\mathbb{S}_{d}^{+}\backslash\{0\}}\left(1-\mathrm{e}^{-\langle\psi(s,u),\xi\rangle}\right)m(\mathrm{d}\xi)\mathrm{d}s\right).
\]
Moreover, if in addition $\int_{\lbrace\Vert\xi\Vert\geq1\rbrace}\Vert\xi\Vert m(\mathrm{d}\xi)<\infty$,
then we infer from Corollary \ref{cor:first moment of pi} that
\[
\lim_{t\to\infty}\mathbb{E}_{x}\left[X_{t}\right]=\int_{0}^{\infty}\mathrm{e}^{s\beta^{\top}}\left(b+\int_{\mathbb{S}_{d}^{+}\backslash\lbrace0\rbrace}\xi m(\mathrm{d}\xi)\right)\mathrm{e}^{s\beta}\mathrm{d}s=\int_{\mathbb{S}_{d}^{+}}y\pi(\mathrm{d}y).
\]

We end this section by considering the following examples.

\begin{exmp}[The matrix-variate basic affine jump-diffusion and
Wishart process] Take $b=2k\Sigma^{\top}\Sigma$ with
$k\geq d-1$ in \eqref{eq:sde affine jump diffusion}. This process
is called matrix-variate basic affine jump-diffusion on $\mathbb{S}_{d}^{+}$
(MBAJD for short), see \cite[Section 4]{MR2956112}. Following
\cite[Section 4.3]{MR2956112}, $\phi(t,u)$ is precisely given by
\[
\phi(t,u)=k\log\det\left(\mathbbm{1}+u\sigma_{t}^{\beta}(\alpha)\right)\int_{0}^{t}\int_{\mathbb{S}_{d}^{+}\backslash\{0\}}\left(1-\mathrm{e}^{-\langle\psi(s,u),\xi\rangle}\right)m(\mathrm{d}\xi)\mathrm{d}s,
\]
and Theorem \ref{thm:existence of limit distribution} implies that
the unique invariant distribution is given by
\begin{align*}
\int_{0}^{\infty}\mathrm{e}^{-\langle u,x\rangle}\pi\left(\mathrm{d}x\right) & =\left(\det\left(\mathbbm{1}+\sigma_{\infty}^{\beta}(\alpha)u\right)\right)^{-k}\exp\left(-\int_{0}^{\infty}\int_{\mathbb{S}_{d}^{+}\backslash\{0\}}\left(1-\mathrm{e}^{-\langle\psi(s,u),\xi\rangle}\right)m(\mathrm{d}\xi)\mathrm{d}s\right),
\end{align*}
where $\sigma_{\infty}^{\beta}(\alpha)=\int_{0}^{\infty}\exp(s\beta)\alpha\exp(s\beta^{\top})\mathrm{d}s$.

The well-known Wishart process, introduced by Bru \cite{MR1132135},
is a special case of the MBAJD with $m=0$. Existence of a unique
distribution was then obtained in \cite[Lemma C.1]{MR3549707}. In
this case $\pi$ is a Wishart distribution with shape parameter $k$
and scale parameter $\sigma_{\infty}^{\beta}(\alpha)$. \end{exmp}

\begin{exmp}[Matrix-variate Ornstein-Uhlenbeck type processes]
For $b=0$ and $\Sigma=0$, we call the solutions to the stochastic
differential equation \eqref{eq:sde affine jump diffusion} matrix-variate
Ornstein-Uhlenbeck (shorted OU) type processes, see \cite{MR2353270}.
Properties of the stationary matrix-variate OU type processes were
investigated in \cite{MR2555198}. Provided $\int_{\lbrace\Vert\xi\Vert\geq1\rbrace}\Vert\xi\Vert m(\mathrm{d}\xi)<\infty$,
Theorem \ref{th:wasserstein ergodicity} implies that the matrix-variate
OU type process is also exponentially ergodic in the Wasserstein-1-distance.
\end{exmp}

\appendix
%dummy comment inserted by tex2lyx to ensure that this paragraph is not empty%dummy comment inserted by tex2lyx to ensure that this paragraph is not empty

\section{Matrix calculus}

\label{sec:matrix calculus}

For a $d\times d$ square matrix $x$, recall that $\mathrm{tr}(x)=\sum_{i=1}^{d}x_{ii}$. The Frobenius norm of $x$ is
given by $\Vert x\Vert=\mathrm{tr}(xx)^{1/2}=(\sum_{i,j=1}^{d}\vert x_{ij}\vert^{2})^{1/2}$.
Let us collect one property of this norm.

\begin{lem}\label{lem:estimate0} Let $x\in\mathbb{S}_{d}^{+}$,
then
\[
\Vert x\Vert\leq\mathrm{tr}(x)\leq\sqrt{d}\Vert x\Vert.
\]
\end{lem}

\begin{proof} Write $x=u^{\top}\kappa u$, where $u$ is orthogonal and $\kappa$ is diagonal with its entries being given by  $\lambda_{i}(x)$,
$i,\ldots,d$,  the eigenvalues of $x$. We have
\[
\Vert x\Vert^{2}=\mathrm{tr}\left(u^{\top}\kappa^{2}u\right)=\sum_{i=1}^{d}\lambda_{i}(x)^{2}.
\]
Since $x\in\mathbb{S}_{d}^{+}$, it holds that $\lambda_{i}(x)\geq0$,
$i=1,\ldots,d$. Then
\[
\left\Vert x\right\Vert =\left(\sum_{i=1}^{d}\lambda_{i}^{2}(x)\right)^{1/2}\leq\sum_{i=1}^{d}\lambda_{i}(x)\leq\sqrt{d}\left(\sum_{i=1}^{d}\lambda_{i}^{2}(x)\right)^{1/2}=\sqrt{d}\Vert x\Vert.
\]
\end{proof}

\bibliographystyle{amsplain}

\def\cprime{$'$} \def\cprime{$'$}
\providecommand{\bysame}{\leavevmode\hbox to3em{\hrulefill}\thinspace}
\providecommand{\MR}{\relax\ifhmode\unskip\space\fi MR }
% \MRhref is called by the amsart/book/proc definition of \MR.
\providecommand{\MRhref}[2]{%
  \href{http://www.ams.org/mathscinet-getitem?mr=#1}{#2}
}
\providecommand{\href}[2]{#2}

\end{document}